\documentclass[]{article}

\usepackage{graphicx,xcolor,textpos}
\usepackage{helvet}
\usepackage[margin = 1in]{geometry}
\usepackage[utf8]{inputenc}
\usepackage[english]{babel}
\usepackage{gensymb}
\usepackage{float}
\usepackage{caption}
\usepackage{subcaption}
\usepackage{amsmath, mathtools, mathrsfs}
\usepackage{amssymb}
\usepackage{amsthm}
\usepackage{enumitem}
\usepackage[makeroom]{cancel}
\usepackage[hidelinks]{hyperref}
\usepackage{tikz}
\usepackage{tikz-cd}
\usepackage{blkarray}
\usepackage[framemethod = default]{mdframed}
\usepackage{multirow}
\usetikzlibrary{calc}
\usepackage{todonotes}
\usepackage{cleveref}
\usetikzlibrary{calc}

\newtheorem{theorem}{Theorem}[section]
\theoremstyle{definition} 
\newtheorem{definition}[theorem]{Definition}
\theoremstyle{plain}
\newtheorem{prop}[theorem]{Proposition}
\theoremstyle{definition} 

\theoremstyle{plain}
\newtheorem{lemma}[theorem]{Lemma}
\newtheorem{cor}[theorem]{Corollary}

\theoremstyle{definition}
\newtheorem{remark}[theorem]{Remark}
\newtheorem{question}[theorem]{Question}

\DeclareMathOperator{\GL}{GL}

\DeclareMathOperator{\Aut}{Aut}

\DeclareMathOperator{\Id}{Id}
\DeclareMathOperator{\diag}{diag}

\DeclareMathOperator{\rank}{rank}

\DeclareMathOperator{\BW}{BW}
\DeclareMathOperator{\supp}{supp}

\DeclareMathOperator{\lcm}{lcm}
\DeclareMathOperator{\Spec}{Spec}
\DeclareMathOperator{\Sp}{Sp}

\newcommand{\Z}{\mathbb{Z}}

\newcommand{\C}{\mathbb{C}}
\newcommand{\Q}{\mathbb{Q}}

\newcommand{\N}{\mathbb{N}}

\newcommand{\E}[1]{\{1,\ldots,#1\}}

\newcommand{\cyclearrow}[2]{
	\coordinate (a) at #1;
	\coordinate (b) at #2;
	\coordinate (v) at ($ (b) - (a) $);
	\coordinate (c) at ($ ($ 0.8*(a) + 0.2*(b) $) + 0.15*([cm={0,1,-1,0,(0,0)}] v) $);
	\coordinate (d) at ($ ($ 0.2*(a) + 0.8*(b) $) + 0.15*([cm={0,1,-1,0,(0,0)}] v) $);
	\draw[-stealth] (c) to[bend left] (d);
}

\newcommand{\picOne}{\begin{tikzpicture}
		\coordinate (va) at (-1,0);
		\coordinate (vb) at (-0.5,0.87);
		\coordinate (vc) at (0.5,0.87);
		\coordinate (vd) at (1,0);
		\coordinate (ve) at (0.5, -0.87);
		\coordinate (vf) at (-0.5, -0.87);
		
		\draw[line width=0.4mm] (va) -- (vd);
		\filldraw [black] (va) circle (2.5pt) node[anchor = east] {$\lambda\,$};
		\filldraw [black] (vb) circle (2.5pt);
		\filldraw [black] (vc) circle (2.5pt);
		\filldraw [black] (vd) circle (2.5pt) node[anchor = west] {$\,\mu$};
		\filldraw [black] (ve) circle (2.5pt);
		\filldraw [black] (vf) circle (2.5pt);
		
		\node at ($ 0.5*(vb) + 0.5*(vc) $) {$\cdots$};
		\node at ($ 0.5*(ve) + 0.5*(vf) $) {$\cdots$};
		\cyclearrow{(va)}{(vb)}
		\cyclearrow{(vc)}{(vd)}
		\cyclearrow{(vd)}{(ve)}
		\cyclearrow{(vf)}{(va)}
		
		\node at (2.5, 0) {or};
		
		\begin{scope}[xshift = 5cm]
			\coordinate (v1) at (0, 1);
			\coordinate (v2) at (0.95, 0.31);
			\coordinate (v3) at (0.59, -0.81);
			\coordinate (v4) at (-0.59, -0.81);
			\coordinate (v5) at (-0.95, 0.31);
			
			\draw[cm={1.5 ,0 ,0 ,1.5 ,(5,0.5)}, line width=0.4mm] (v1)  to[in=50,out=130, loop] (v1);
			
			\filldraw [black] (v1) circle (2.5pt) node[anchor = north, yshift=-3] {$\lambda\,$};
			\filldraw [black] (v2) circle (2.5pt);
			\filldraw [black] (v3) circle (2.5pt);
			\filldraw [black] (v4) circle (2.5pt);
			\filldraw [black] (v5) circle (2.5pt);

			\cyclearrow{(v1)}{(v2)}
			\cyclearrow{(v2)}{(v3)}
			\cyclearrow{(v4)}{(v5)}
			\cyclearrow{(v5)}{(v1)}
			
			\node at ($ 0.5*(v3) + 0.5*(v4) $) {$\cdots$};
		\end{scope}
		
\end{tikzpicture}}

\newcommand{\picTwo}{\begin{tikzpicture}
		\coordinate (v1) at (-0.5, 0);
		\coordinate (v2) at (-1.19, -0.95);
		\coordinate (v3) at (-2.31, -0.59);
		\coordinate (v4) at (-2.31, 0.59);
		\coordinate (v5) at (-1.19, 0.95);
		
		\coordinate (v6) at (0.5, 0);
		\coordinate (v7) at (1.19, 0.95);
		\coordinate (v8) at (2.31, 0.59);
		\coordinate (v9) at (2.31, -0.59);
		\coordinate (v10) at (1.19, -0.95);
		
		\draw[line width=0.4mm] (v1) -- (v6);
		\filldraw [black] (v1) circle (2.5pt) node[anchor = east] {$\lambda\,$};
		\filldraw [black] (v2) circle (2.5pt);
		\filldraw [black] (v3) circle (2.5pt);
		\filldraw [black] (v4) circle (2.5pt);
		\filldraw [black] (v5) circle (2.5pt);

		\cyclearrow{(v1)}{(v2)}
		\cyclearrow{(v2)}{(v3)}
		\cyclearrow{(v4)}{(v5)}
		\cyclearrow{(v5)}{(v1)}
		
		\filldraw [black] (v6) circle (2.5pt) node[anchor = west] {$\,\mu$};
		\filldraw [black] (v7) circle (2.5pt);
		\filldraw [black] (v8) circle (2.5pt);
		\filldraw [black] (v9) circle (2.5pt);
		\filldraw [black] (v10) circle (2.5pt);
		
		\cyclearrow{(v6)}{(v7)}
		\cyclearrow{(v7)}{(v8)}
		\cyclearrow{(v9)}{(v10)}
		\cyclearrow{(v10)}{(v6)}
		
		\node at ($ 0.4*(v3) + 0.6*(v4) $) {$\vdots$};
		\node at ($ 0.6*(v8) + 0.4*(v9) $) {$\vdots$};
\end{tikzpicture}}

\title{The $R_\infty$-property for right-angled Artin groups and their nilpotent quotients}
\author{Thomas Witdouck \thanks{KU Leuven Campus Kortrijk Kulak, Department of Mathematics, Etienne Sabbelaan 53, 8500 Kortrijk, Belgium. Email: thomas.witdouck@kuleuven.be. The author was supported by a PhD fellowship of the Research Fund - Flanders (FWO), Grant Number 1153122N.}}

\begin{document}

\maketitle

\begin{abstract}
	It is proven that every non-abelian right-angled Artin group has the $R_\infty$-property and bounds are given on the $R_\infty$-nilpotency index. In case the graph is transposition-free, which is true for almost all graphs, it is shown that the $R_\infty$-nilpotency index is equal to 2.
\end{abstract}

\section{Introduction and results}

Let $G$ be a group. For any automorphism $\varphi$ of $G$, one says $x,y \in G$ are $\varphi$-conjugate if and only if there exists an element $z \in G$ such that $x = z y \varphi(z)^{-1}$. Being $\varphi$-conjugate defines an equivalence relation on $G$ and the number of equivalence classes is called the Reidemeister number of $\varphi$. This number is denoted by $R(\varphi)$ and takes values in $\N_0 \cup \{\infty\}$. \footnote{In this paper we use the convention that $\N = \{0, 1, 2, \ldots\}$ and $\N_0 = \N \setminus \{0\}$.} The \textit{Reidemeister spectrum} of a group $G$, written $\Spec_R(G)$, is defined as the set of all Reidemeister numbers $\Spec_R(G) = \{ R(\varphi) \mid \varphi \in \Aut(G) \}$.
\begin{definition}
	A group $G$ is said to have the \textit{$R_\infty$-property} if $\Spec_R(G) = \{\infty\}$.
\end{definition}
The Reidemeister number finds its origin in Nielsen fixed point theory. Under certain assumptions, the Reidemeister spectrum of the fundamental group of a topological space $X$ can give  information about the number of fixed points of homeomorphisms on $X$. For example, if the topological space is a nilmanifold, i.e. a quotient $N/L$ of a simply connected nilpotent Lie group $N$ by a cocompact lattice $L$, and its fundamental group (which is isomorphic to $L$) has the $R_\infty$-property, then every homeomorphism of $N/L$ is homotopic to a map which has no fixed points. The fundamental groups of nilmanifolds are exactly the finitely generated torsion-free nilpotent groups. Nilpotent quotients of right-angled Artin groups fall under this class and are the groups that will be considered in this paper. A natural question for these groups is: how does the $R_\infty$-property behave with respect to the nilpotency class? In this light, the following definition was introduced. We let $\gamma_{i}(G)$ denote the $i$-th term in the lower central series of $G$, i.e. $\gamma_1(G)= G$ and $\gamma_{i+1}(G) = [G, \gamma_i(G)]$. The same notation will be used for Lie algebras.

\begin{definition}
	Let $G$ be a group. The \textit{$R_\infty$-nilpotency index} is the least integer $c$ such that $G/\gamma_{c+1}(G)$ has the $R_\infty$-property. If no such integer exists, then we say the $R_\infty$-nilpotency index is infinite. 
\end{definition}

The $R_\infty$-nilpotency index was first introduced in \cite{dg16-1} and determined for surface groups. The index has also been studied on other groups, e.g. free groups \cite{dg14-1} and Baumslag-Solitar groups \cite{dg20-1}.

\begin{remark}
	\label{rem:RinftyNilpIndex}
	If a group $G$ has finite $R_\infty$-nilpotency index $c$, then $G$ itself and the groups $G/\gamma_{i+1}(G)$ for any integer $i \geq c$ also have the $R_\infty$-property. This follows from Lemma 1.1 in \cite{gw09-1}.
\end{remark}

In this paper, a graph $\Gamma$ is defined as a tuple $(V, E)$ with $V$ a finite set and $E$ a subset of $\{ \{v, w\} \,|\, v,w \in V \text{ and } v \neq w \}$. The elements in $V$ are called the vertices and the elements in $E$ the edges. To this information, one can associate a group $A(\Gamma)$ with presentation:
\[ A(\Gamma) = \Big\langle V \; \Big| \; [v, w]; \: v,w \in V, \,\{v, w\} \notin E \Big\rangle. \]
In the literature, these groups are referred to as: right-angled Artin groups (RAAG for short), (free) partially commutative groups, graph groups, etc. and they are also often defined using the opposite convention, namely that vertices commute if and only if they are connected with an edge in the graph. The purpose of this paper is to study the $R_\infty$-property and the $R_\infty$-nilpotency index of the groups $A(\Gamma)$ for an arbitrary graph $\Gamma$. To this extend, we also define for any integer $c > 1$ the $c$-step nilpotent quotient of $A(\Gamma)$ as
\[ A(\Gamma, c) = \frac{A(\Gamma)}{\gamma_{c+1}(A(\Gamma))}. \]
In \cite{send21-1} it was proven that $A(\Gamma)$ has the $R_\infty$-property for certain subclasses of graphs, among which the class of non-empty transvection-free graphs, where a graph is called \textit{non-empty} if its set of edges is non-empty and \textit{transvection-free} if $v \mapsto vw$ does not define an automorphism of $A(\Gamma)$ for any vertices $v,w \in V$. Moreover, from the proof it follows that the $R_\infty$-nilpotency index is equal to either $2$ or $3$. In section \ref{sec:transposition-freeGraphs} we improve this result by showing Theorem \ref{thm:transpositionFreeIndex2} as stated below. A graph $\Gamma$ is called \textit{transposition-free} if the transposition of no two distinct vertices yields a graph automorphism. Being transposition-free is a weaker condition than being transvection-free.
	
\begin{theorem}
	\label{thm:transpositionFreeIndex2}
	If $\Gamma$ is a non-empty transposition-free graph, then $A(\Gamma)$ has $R_\infty$-nilpotency index 2.
\end{theorem}
	
Let $\mathcal{P}$ be a property of graphs, invariant under isomorphism. Let $G_n$ denote the total number of isomorphism classes of graphs on $n$ vertices and $G_n(\mathcal{P})$ the total number of isomorphism classes of graphs on $n$ vertices which have property $\mathcal{P}$. The property $\mathcal{P}$ is said to hold for almost all unlabelled graphs if
\[ \lim_{n \to \infty} \frac{G_n(\mathcal{P})}{G_n} = 1. \]
The term `unlabelled' refers to the fact that we count the graphs up to isomorphism. Since two right-angled Artin groups are isomorphic if and only if the corresponding graphs are isomorphic, it makes sense to count the graphs up to isomorphism in this context. As a consequence of a classical result of Erd\"os and R\'enyi on labelled graphs \cite{er63-1}, it is argued in section 1.6 of \cite{baba95-1} that almost all unlabelled graphs are asymmetric, i.e. have a trivial automorphism group. Since an asymmetric graph is also transposition-free, we immediately get the following corollary.
\begin{cor}
	The group $A(\Gamma)$ has $R_\infty$-nilpotency index 2 for almost all unlabelled graphs $\Gamma$.
\end{cor}

Analogous to the partially commutative groups, one can define partially commutative Lie algebras, which will play a central role in this paper. Let $\Gamma = (V, E)$ be a graph, $K$ a field, $\mathfrak{f}^K(V)$ the free Lie algebra on $V$ over $K$ and $I(\Gamma)$ the ideal of $\mathfrak{f}^K(V)$ generated by the set of Lie brackets $\{ [v, w] \mid v, w \in V,\,\{v, w\} \notin E \}$. The partially commutative Lie algebra $L^K(\Gamma)$ associated to the graph $\Gamma$ over the field $K$ is then defined as the quotient
\[ L^K(\Gamma) := \frac{\mathfrak{f}^K(V)}{I(\Gamma)}. \]
For any integer $c > 1$ the free $c$-step nilpotent partially commutative Lie algebra is then defined as
\[ L^K(\Gamma, c) = \frac{L^K(\Gamma)}{\gamma_{c+1}(L^K(\Gamma))}. \]
By the work in \cite{dm05-1} and \cite{dm21-1}, the automorphism group of $L^K(\Gamma, 2)$ was completely determined for $K$ any subfield of $\C$. Later, this result was used in \cite{dw22-1} to determine the group of graded automorphisms of $L^K(\Gamma)$ and $L^K(\Gamma, c)$ for any $c>1$.

Let $\Gamma = (V, E)$ be a graph. An equivalence relation on $V$ can be defined by saying two vertices are equivalent if and only if their transposition yields a graph automorphism. The equivalence classes are called the \textit{coherent components} and $\Lambda$ denotes the set of all coherent components. This equivalence relation gives rise to a quotient graph $\overline{\Gamma} = (\Lambda, \overline{E}, \Phi)$ of which the exact definition is given in section \ref{sec:gradedAutomorphisms}. The $R_\infty$-nilpotency index strongly depends on the sizes of adjacent coherent components in the quotient graph $\overline{\Gamma}$. To this extend we define for any non-empty graph $\Gamma$, the positive integers
\begin{equation}
	\label{eq:lowerBound}
	\xi(\Gamma) = \min \left\{ |\lambda| + |\mu| \: \middle| \: \lambda, \mu \in \Lambda,\, \{\lambda , \mu\} \in \overline{E} \right\}
\end{equation}
and
\begin{equation}
	\label{eq:upperBound}
	\Xi(\Gamma) = \min \left\{ c(\lambda, \mu) \: \middle| \: \lambda, \mu \in \Lambda, \, \{\lambda, \mu \} \in \overline{E} \right\}
\end{equation}
where for any $\lambda, \mu \in \Lambda$:
\begin{equation}
	\label{eq:clambdamu}
	c(\lambda, \mu) = \begin{cases}
		\max\{ 2|\lambda| + |\mu|, |\lambda| + 2|\mu| \} \quad & \text{if } \lambda \neq \mu\\
		2|\lambda| & \text{if } \lambda = \mu.
	\end{cases}
\end{equation}
Note that given an arbitrary graph $\Gamma$, these numbers are relatively easy to compute as the time complexity is polynomial with respect to the number of vertices. The main theorem of this paper is then formulated as follows:
\begin{theorem}
	\label{thm:allRAAGSRinfty}
	If $\Gamma$ is a non-empty graph, then $A(\Gamma)$ has the $R_\infty$-property with $R_\infty$-nilpotency index $c$ satisfying \[\xi(\Gamma) \leq c \leq \Xi(\Gamma).\]
\end{theorem}
Note that a graph $\Gamma$ is transposition-free if and only if each coherent component has size one. In this case, the above theorem thus tells us that a transposition-free graph has $R_\infty$-nilpotency index $c$ equal to $2$ or $3$. Therefore Theorem \ref{thm:transpositionFreeIndex2} is really a stronger statement than Theorem \ref{thm:allRAAGSRinfty} in the transposition-free case.

As a second remark, note that the upper bound $\Xi(\Gamma)$ is always less or equal than $2r$ with $r$ the number of vertices of the graph. We thus get that the $R_\infty$-nilpotency index of $A(\Gamma)$ is always less or equal than $2r$. This agrees with the known result for free groups, which states that the $R_\infty$-nilpotency index of a free group of finite rank at least 2 is equal to two times its rank \cite{dg14-1}. Indeed, the free group of rank $r$ is isomorphic to $A(\Gamma)$ with $\Gamma$ the complete graph on $r$ vertices. This graph has only one coherent component of size $r$ and thus $\xi(\Gamma) = \Xi(\Gamma) = 2r$.

\paragraph{Acknowledgements}
The author would like to thank Maarten Lathouwers and Pieter Senden for the useful discussions on the subject.

\section{From group to graded Lie Algebra}
\label{sec:fromGroupToGradedLieAlg}
In this section, we recall a well-known construction which associates to any group a graded Lie ring. By taking the tensor product with a field $K$, we get a graded Lie algebra. It is a classical result that if one does this construction for the group $A(\Gamma)$, one obtains exactly the Lie algebra $L^K(\Gamma)$. For finitely generated nilpotent groups, this construction is particularly useful in this paper since the $R_\infty$-property on the group can be characterized by an eigenvalue condition on the induced automorphisms on the graded Lie algebra. First, let us recall what we mean by a graded Lie algebra.

\begin{definition}
	Let $L$ be a Lie algebra. We say $L$ is \textit{positively graded} if there exists vector subspaces $L_i \subset L$ for all integers $i > 0$ such that $L = \bigoplus_{i = 1}^\infty L_i$ and for any integers $i,j > 0$ it holds that $[L_i, L_j] \subset L_{i + j}$. In case $L$ is positively graded and we have fixed a positive grading, an automorphism $\varphi$ of $L$ will be called \textit{graded} if $\varphi(L_i) = L_i$ for all integers $i \geq 0$. We denote the subgroup of graded automorphisms of $L$ by $\Aut_g(L)$.
\end{definition}

Let us write $\Sp_K(S)$ for the vector space span over the field $K$ of some subset $S$ of a vector space and let $\Gamma =(V, E)$ be a graph. The free Lie algebra on the vertices $V$ is canonically positively graded by
\[ \mathfrak{f}^K(V) = \bigoplus_{i = 1}^\infty \mathfrak{f}_i^K(V) \]
where the subspace $\mathfrak{f}^K_i(V)$ are defined inductively by \[\mathfrak{f}_1^K(V) = \Sp_K(V) \quad  \text{ and } \quad \mathfrak{f}_{i+1}^K(V) = [\mathfrak{f}_1^K(V), \mathfrak{f}_i^K(V)]. \]
Let $p_1: \mathfrak{f}^K(V) \to L^K(\Gamma)$ and $p_2: L^K(\Gamma) \to L^K(\Gamma, c)$ denote the projection homomorphisms. The free (nilpotent) partially commutative Lie algebras $L^K(\Gamma)$ and $L^K(\Gamma, c)$ as defined in the introduction are positively graded as well by
\[ L^K(\Gamma) = \bigoplus_{i = 1}^\infty L^K_i(\Gamma) \quad \text{and} \quad L^K(\Gamma, c) = \bigoplus_{i = 1}^c L_i^K(\Gamma, c) \]
where
\[ L_i^K(\Gamma) = p_1(\mathfrak{f}_i^K(V)) \quad \text{and} \quad L_i^K(\Gamma, c) = (p_2 \circ p_1)(\mathfrak{f}_i^K(V)). \]
Note that $L_i^K(\Gamma, c) = \{0\}$ for any $i > c$.

Next, let us recall the classical construction which associates to any group a positively graded Lie algebra. Let $G$ be a group. For any $i \geq 1$, the quotient group $L_i(G) = \gamma_{i}(G)/\gamma_{i+1}(G)$ is abelian and thus it carries a $\Z$-module structure. The Lie ring $L(G)$ is defined as the direct sum of these $\Z$-modules:
\[ L(G) := \bigoplus_{i = 1}^{\infty} L_i(G), \]
together with a Lie bracket
\[ [g\, \gamma_{i+1}(G), h\, \gamma_{j+1}(G)] := [g,h] \gamma_{i+j+1}(G)\]
for any $i,j \geq 1, \, g \in \gamma_{i}(G)$ and $h \in \gamma_{j}(G)$. Details of this construction can for example be found in section 2 of \cite{wade15-1}. By taking the tensor product of $L(G)$ with a field $K$, we get a Lie algebra $L^K(G) := L(G) \otimes_\Z K = \bigoplus_{i = 1}^\infty L_i(G) \otimes_\Z K$ where we will write $L^K_i(G) := L_i(G) \otimes_\Z K$. Note that the Lie algebra $L^K(G)$ can be infinite dimensional, but if $G$ is finitely generated nilpotent $L^K(G)$ necessarily has finite dimension. It is a classical result that the lower central series of the Lie algebra $L^K(G)$ can be expressed as
\[ \gamma_j(L^K(G)) = \bigoplus_{i = j}^\infty L_i^K(G). \]

Clearly, $L^K(G)$ is a positively graded Lie algebra, where the grading is given by the subspaces $L_i^K(G)$. Note that any automorphism $\varphi$ of the group $G$ induces $\Z$-module automorphisms $\varphi_i:L_i(G) \to L_i(G)$ for all $i \geq 1$. On the tensor product with $K$, we thus get for each $i$ a linear map $\overline{\varphi_i} = \varphi_i \otimes \Id :L_i^K(G) \to L_i^K(G)$. Combining these automorphisms $\overline{\varphi}_i$, we get a graded Lie algebra automorphism $\overline{\varphi} \in \Aut_g(L^K(G))$ defined by
\[ \overline{\varphi}:L^K(G) \to L^K(G): x \mapsto \overline{\varphi_i}(x) \quad \forall x \in L_i^K(G).  \]

\begin{remark}
	\label{rem:charPolyIntegral}
	Let $H$ be a finitely generated abelian group and write $\tau(H)$ for its torsion subgroup which is characteristic in $H$. Let the field $K$ be of characteristic 0. Then there is a natural isomorphism between $\frac{H}{\tau(H)} \otimes_\Z K$ and $H \otimes_\Z K$. If $f$ is an automorphism of $H$, it also induces an automorphism on the free $\Z$-module $H/\tau(H)$ which can be represented by a matrix in $\GL(n,\Z)$ with respect to some basis for $H/\tau(H)$. The induced linear map on $\frac{H}{\tau(H)} \otimes_\Z K \cong H \otimes_\Z K$ can therefore be represented by the same matrix and will thus have a characteristic polynomial with integer coefficients and constant term equal to $\pm 1$. Applying this to the abelian groups $L_i^K(G)$ from above for a field $K$ of characteristic 0, we see that for all $i > 0$, $\overline{\varphi_i}$ has characteristic polynomial with coefficients in $\Z$ and constant term equal to $\pm 1$.
\end{remark}

The following lemma motivates why one would look at the induced automorphism on the Lie algebra when studying the $R_\infty$-property on finitely generated nilpotent groups. It was stated in \cite{dg14-1} as a generalization of corollary 4.2 from \cite{roma11-1}.
\begin{lemma}
	\label{lem:reidemeisterNumberFinGenNilpGroup}
	Let $G$ be a finitely generated nilpotent group and $\varphi \in \Aut(G)$. Then $R(\varphi) = \infty$ if and only if 1 is an eigenvalue of $\overline{\varphi} \in \Aut_g(L^\C(G))$.
\end{lemma}

In \cite{kd92-1} it was proven that the Lie algebra $L^K(A(\Gamma))$ is isomorphic to the partially commutative Lie algebra $L^K(\Gamma)$. An explicit isomorphism is given by sending the vertices in $L^K(\Gamma)$ to the corresponding vertices in $L^K\left(A(\Gamma)\right)$. Moreover, this isomorphism preserves the grading of both Lie algebras. By consequence, we get for all integers $c > 1$ the isomorphisms
\[ L^K (\Gamma, c) \cong \frac{L^K\left(A(\Gamma)\right)}{\gamma_{c+1}\left(L^K\left(A(\Gamma)\right)\right)} \cong L^K\left(A(\Gamma, c)\right), \]
where each isomorphism is induced by the identity on the vertices of $\Gamma$. Thus, if we want to apply Lemma \ref{lem:reidemeisterNumberFinGenNilpGroup} to the group $A(\Gamma, c)$, it will be useful to know what the group $\Aut_g\left(L^K (\Gamma, c)\right)$ looks like. Since every element in this group maps $\Sp_K(V)$ onto itself, we get a group homomorphism
\[ \pi_c:\Aut_g(L^K (\Gamma, c)) \to \GL(\Sp_K(V)): \varphi \mapsto \varphi|_{\Sp_K(V)}. \]
It is well known that $\Sp_K(V)$ generates $L^K (\Gamma, c)$ as a Lie algebra. As a consequence, any graded automorphism of $L^K (\Gamma, c)$ is completely determined by its induced map on the abelianization and thus $\pi_c$ is a group isomorphism onto its image. In the next section we discuss the image of $\pi_c$.

\section{The group of graded automorphisms of $L^K(\Gamma,c)$}
\label{sec:gradedAutomorphisms}

In this section we state a known result which describes the group of graded automorphisms of $L^K(\Gamma,c)$ for $K$ any subfield of $\C$. Before we can do so, we need to introduce some definitions. Consider a graph $\Gamma = (V, E)$. For any vertex $v \in V$, the open and closed neighbourhoods of $v$ are defined as
\begin{equation}
	\label{eq:neighbourhoodVertex}
	\Omega'(v) = \{ w \in V \mid \{w, v\} \in E \} \quad \text{and} \quad \Omega(v) = \Omega'(v) \cup \{v\},
\end{equation}
respectively. A relation $\prec$ on $V$ is defined as
\[ v \prec w \: \Leftrightarrow \: \Omega'(v) \subseteq \Omega(w). \]
An equivalence relation $\sim$ is then defined as
\[ v \sim w \: \Leftrightarrow \: v \prec w \text{ and } v \succ w. \]
The equivalence classes are referred to as the \textit{coherent components} and the set of coherent components is denoted by $\Lambda = V/\sim$. Note that $v \sim w$ is equivalent to saying that the transposition of $v$ and $w$ yields a graph automorphism of $\Gamma$.
\begin{remark}
	\label{rem:quotientgraph}
	Let $\lambda, \mu \in \Lambda$ be two not necessarily distinct coherent components. It is straightforward to verify that if there exist $v \in \lambda$, $w \in \mu$ such that $\{v, w\} \in E$ then for any $v' \in \lambda$ and $w' \in \mu$ such that $v' \neq w'$ it holds that $\{v', w'\} \in E$.
\end{remark}

The \textit{quotient graph} of $\Gamma$ is defined as the 3-tuple $\overline{\Gamma} = (\Lambda, \overline{E}, \Psi)$ where \[\overline{E} = \{ \{\lambda, \mu\} \mid \exists v \in \lambda, \, w \in \mu:\: \{v,w\} \in E \} \quad \text{and} \quad \Psi:\Lambda \to \N: \lambda \mapsto |\lambda|.\]
Note that the quotient graph can have \textit{loops}, i.e singletons in $\overline{E}$. An automorphism of the quotient graph $\overline{\Gamma}$ is a bijection $\varphi:\Lambda \to \Lambda$ which satisfies $\varphi(\overline{E}) = \overline{E}$ and $\Psi \circ \varphi = \Psi$. The group of all automorphisms of $\overline{\Gamma}$ is written as $\Aut(\overline{\Gamma})$. It is clear that any automorphism $\varphi \in \Aut(\Gamma)$ induces an automorphism $p(\varphi) \in \Aut(\overline{\Gamma})$, by letting $p(\varphi)(\lambda) = \varphi(\lambda)$ for any $\lambda \in \Lambda$. This gives a group morphism $p$, which fits in the short exact sequence
\begin{equation*}
	1  \longrightarrow \prod_{\lambda \in \Lambda} S(\lambda) \longrightarrow \Aut(\Gamma) \stackrel{p}{\longrightarrow} \Aut(\overline{\Gamma}) \longrightarrow 1,
\end{equation*}
where $S(X)$ denotes the symmetric group on the set $X$. Moreover, this sequence is right split. With other words, there exists a group morphism $r:\Aut(\overline{\Gamma}) \to \Aut(\Gamma)$ such that $p \circ r = \Id$. We construct such a morphism $r$ by first fixing an order on the vertices inside each coherent component $\lambda = \{ v_{\lambda 1}, \ldots , v_{\lambda \Psi(\lambda)} \}$ and then defining for any $\varphi \in \Aut(\overline{\Gamma})$ the automorphism $r(\varphi) \in \Aut(\Gamma)$ by $r(\varphi)(v_{\lambda i}) = v_{\varphi(\lambda) i}$ for any $\lambda \in \Lambda$ and $i \in \{1, \ldots, \Psi(\lambda)\}$. We fix this order on the vertices inside each coherent component and this morphism $r$ for the remainder of this paper.
For any vertices $v,w \in V$, let $E_{vw}$ denote the linear endomorphism on $\Sp_K(V)$ determined by $E_{vw}(w) = v$ and $E_{vw}(v') = 0$ for any $v' \in V$ with $v' \neq w$. We define the subgroup $\mathcal{U} \leq \GL(\Sp_K(V))$ as
\[ \mathcal{U} := \left\langle I + tE_{vw} \: \middle| \: v \prec w,\, v \not\sim w, \, t \in K \right\rangle, \]
where $I$ denotes the identity map on $\Sp_K(V)$.
For any graph automorphism $\varphi \in \Aut(\Gamma)$ let us write $P_\varphi$ for the linear map on $\Sp_K(V)$ defined by $P_\varphi(v) = \varphi(v)$ for any $v \in V$. This gives a group morphism $P:\Aut(\Gamma) \to \GL(\Sp_K(V)): \varphi \mapsto P_\varphi$. One can also compose this morphism with the morphism $r$ from before, giving a group morphism
\[\overline{P} := P \circ r : \Aut(\overline{\Gamma}) \to \GL(\Sp_K(V)).\] 
For any coherent component $\lambda \in \Lambda$, we identify the group $\GL(\Sp_K(\lambda))$ with a subgroup of $\GL(\Sp_K(V))$ where a linear map $f \in \GL(\Sp_K(\lambda))$ is identified with the linear map in $\GL(\Sp_K(V))$ which maps a vertex $v \in V$ to itself if $v \notin \lambda$ and $v$ to $f(v)$ if $v \in \lambda$.

The following was proven in \cite{dm21-1} for $c = 2$ and later generalized in \cite{dw22-1} to $c > 2$.

\begin{theorem}
	\label{thm:gradedAutoGroup}
	For any $c > 1$ and any field $K \subset \C$, the image of $\pi_c$ in $\GL(\Sp_K(V))$ is a linear algebraic group given by
	\[ G^K(\Gamma) := \overline{P}(\Aut(\overline{\Gamma})) \cdot \left( \prod_{\lambda \in \Lambda} \GL(\Sp_K(\lambda)) \right) \cdot \mathcal{U} \]
	where $\mathcal{U}$ is the unipotent radical of $G^K(\Gamma)$.
\end{theorem}

\begin{remark}
	\label{rem:semisimplePart}
	Let $K$ be any subfield of $\C$ and $g \in G^K(\Gamma)$ a semi-simple element, i.e. an element whose matrix w.r.t. the basis $V$ is diagonalizable over $\C$. Then $g$ lies in some linearly reductive subgroup of the linear algebraic group $G^K(\Gamma)$. Note as well that
	\[\mathcal{R} := \overline{P}(\Aut(\overline{\Gamma})) \cdot \prod_{\lambda \in \Lambda} \GL(\Sp_K(\lambda))\]
	is a linearly reductive subgroup of $G^K(\Gamma)$ such that $G^K(\Gamma) \cong \mathcal{R} \ltimes \mathcal{U}$. It thus follows from Theorem 4.3. in chapter VIII in \cite{hoch81-1} that there exists an $h \in G^K(\Gamma)$ such that $h g h^{-1} \in \mathcal{R}$. In particular, this is valid for $K = \Q$ which will be of use in the next section.
\end{remark}

\subsection{Automorphisms with the same characteristic polynomial}

The goal of this section is to find for any automorphism $\varphi$ of the group $A(\Gamma, c)$ an automorphism $f \in \Aut_g(L^\C(\Gamma, c))$ for which it is clear what the eigenvectors and eigenvalues on $L_1^\C(\Gamma, c)$ are and such that $f$ and the induced automorphism $\overline{\varphi} \in \Aut_g(L^\C(\Gamma, c))$ have the same characteristic polynomial. Moreover, the fact that $\overline{\varphi}$ has characteristic polynomial with integer coefficients and constant term equal to $\pm 1$ (see Remark \ref{rem:charPolyIntegral}), will give us relations among the eigenvalues of $f$ on $L_1^\C(\Gamma, c)$. At the end of the section, we also prove a corollary of these results which says that any automorphism of the rational Lie algebra $L^\Q(\Gamma, c)$ which has characteristic polynomial with coefficients in $\Z$ and determinant equal to $\pm 1$ has a characteristic polynomial which can be realized as the characteristic polynomial of an automorphism of the group $A(\Gamma, c)$.

For notational purposes, we introduce the following terminology.

\begin{definition}
	A linear map or matrix over a field $K \subset \C$ will be called \textit{integer-like} if its characteristic polynomial has integer coefficients and constant term equal to $\pm 1$.
\end{definition}

By considering the Frobenius normal form or rational canonical form of a matrix, it is clear that a matrix in $\GL(n, K)$ is integer-like if and only if it is conjugated to a matrix in $\GL(n, \Z)$ by a matrix in $\GL(n, K)$. From this observation it follows that powers of integer-like matrices remain integer-like. For any subfield $K \subset \C$, any integers $k > 1$, $n_1, \ldots, n_k > 0$ and any matrices $B_i \in \GL(n_i, K)$, let $\diag(B_1, \ldots, B_n)$ denote the block-diagonal matrix with the $B_i$'s on the diagonal. Using the fact that the eigenvalues of integer-like matrices are algebraic units, the following can be proven about integer-like block-diagonal matrices over $\Q$.

\begin{lemma}
	\label{lem:integerLikeBlockDiagonal}
	Let $k$ and $n_1, \ldots, n_k$ be positive integers and $B_i \in \GL(n_i, \Q)$ matrices and let $B = \diag(B_1, \ldots, B_k)$. If $B$ is integer-like, then so is every block $B_i$.
\end{lemma}

Let $\Gamma = (V, E)$ be a graph and recall that in the previous section we fixed an order on the vertices inside each coherent component as $\lambda = \{ v_{\lambda 1}, \ldots , v_{\lambda |\lambda|} \}$. For any integer $n > 0$, let $I_n$ denote the $n \times n$ identity matrix. For any $k$-cycle $\sigma = (\lambda_1 \ldots \lambda_k) \in S(\Lambda)$ on the coherent components with $|\lambda_1| = \ldots = |\lambda_k| = n$ and any $B \in \GL(n, K)$, we define the linear map $T(\sigma, B) \in \GL(\Sp_K(V))$ as represented by the matrix $B$ if $k = 1$ and by the matrix
\[ \begin{pmatrix}
0 & & & B\\
I_n & 0 & &\\
& \ddots & \ddots & \\
& & I_n & 0
\end{pmatrix}\]
if $k > 1$ with respect to the vertices $v_{\lambda_1 1}, \ldots, v_{\lambda_1 n}, v_{\lambda_2 1}, \ldots, v_{\lambda_2 n}, \ldots, v_{\lambda_k 1},\ldots, v_{\lambda_k n}$ and the identity on all other vertices. Let $\psi \in \Aut(\overline{\Gamma})$ be an automorphism of the quotient graph with disjoint cycle decomposition $\psi = \sigma_1 \circ \ldots \circ \sigma_d$ (where 1-cycles are not omitted) and write $n_i$ for the size of the coherent components in the cycle $\sigma_i$. Note that for any matrices $B_i \in \GL(n_i, K)$, the linear map $\prod_{i = 1}^d T(\sigma_i, B_i)$ is an element of $G^K(\Gamma)$.

\begin{lemma}
	\label{lem:reducedFormOverQ}
	Let $\Gamma = (V, E)$ be a graph and $g \in G^\Q(\Gamma)$ a semi-simple integer-like element. Then there exists a $\psi \in \Aut(\overline{\Gamma})$ with disjoint cycle decomposition $\psi = \sigma_1 \circ \ldots \circ \sigma_d$ (where 1-cycles are not omitted) and for each cycle $\sigma_i$ consisting of coherent components of size $n_i$, an integral semi-simple matrix $B_i \in \GL(n_i, \Z)$ such that $g$ is conjugated to $\prod_{i = 1}^d T(\sigma_i, B_i)$ in $G^\Q(\Gamma)$.
\end{lemma}
\begin{proof}
	By Theorem \ref{thm:gradedAutoGroup} and Remark \ref{rem:semisimplePart}, there exists a $\psi \in \Aut(\overline{\Gamma})$ and a $C \in  \prod_{\lambda \in \Lambda} \GL(\Sp_\Q(\lambda))$ such that $g$ is conjugated to $\overline{P}(\psi) C$ in $G^\Q (\Gamma)$. As a consequence, $\overline{P}(\psi) C$ is also semi-simple and integer-like. Let $\psi = \sigma_1 \circ \ldots \circ \sigma_d$ be a disjoint cycle decomposition of $\psi$ with $\sigma_i = (\lambda_{i1}, \ldots, \lambda_{ik_i})$ and $|\lambda_{i1}| = \ldots = |\lambda_{ik_i}| = n_i$. For any $i \in \E{d}$, define the subspaces $W_i = \Sp_\Q(\lambda_{i1} \cup \ldots \cup \lambda_{ik_i})$. It follows that each $W_i$ is invariant under $\overline{P}(\psi) C$. Define $h_i := (\overline{P}(\psi) C)|_{W_i}$. It follows that for each $i \in \E{d}$, $h_i$ is semi-simple and using Lemma \ref{lem:integerLikeBlockDiagonal} that $h_i$ is integer-like.
	
	Note that for any $\lambda \in \{ \lambda_{i1}, \ldots, \lambda_{ik_i} \}$ we have $h_i(\Sp_\Q(\lambda)) = \Sp_\Q(\sigma_i(\lambda))$. By consequence we get that the set $\lambda_{i1} \cup h_i(\lambda_{i1}) \cup \ldots \cup h_i^{k_i -1}(\lambda_{i1})$ is a basis for $W_i$. Clearly, with respect to this basis, $h_i$ is represented by a matrix $\tilde{B}_i \in \GL(n_i, \Q)$ if $k_i =1$ and by the matrix
	\[ \begin{pmatrix}
		0 & & & \tilde{B}_i\\
		I_{n_i} & 0 & &\\
		& \ddots & \ddots & \\
		& & I_{n_i} & 0
	\end{pmatrix}\]
	if $k_i > 1$ for some $\tilde{B}_i \in \GL(n_i, \Q)$. This is equivalent to saying that $\overline{P}(\psi) C$ is conjugated to $\prod_{i = 1}^d T(\sigma_i, \tilde{B}_i)$ by an element of $\prod_{\lambda \in \Lambda} \GL(\Sp_\Q(\lambda)) \subset G^\Q(\Gamma)$. Again, we have that each $T(\sigma_i, \tilde{B}_i)|_{W_i}$ is semi-simple and integer-like and the same must hold for its $k_i$'th power $\left( T(\sigma_i, \tilde{B}_i)|_{W_i} \right)^{k_i}$ which is represented by the matrix $\diag(\tilde{B}_i, \ldots, \tilde{B}_i)$ with respect to the basis $\lambda_{i1} \cup \ldots \cup \lambda_{ik_i}$. We thus get that $\tilde{B}_i$ is semi-simple and using Lemma \ref{lem:integerLikeBlockDiagonal} that $\tilde{B}_i$ is integer-like. By consequence, there exists a matrix $Q_i \in \GL(n_i, \Q)$ such that $B_i := Q_i \tilde{B}_i Q_i^{-1} \in \GL(n, \Z)$. Clearly $B_i$ is also semi-simple. Define $Q \in \prod_{\lambda \in \Lambda} \GL(\Sp_\Q(\lambda)) \subset G^\Q(\Gamma)$ as represented by the matrix $\diag(Q_i, \ldots, Q_i)$ with respect to the basis $\lambda_{i1} \cup \ldots \cup \lambda_{ik_i}$ on the subspace $W_i$ for each $i \in \E{d}$. It is then straightforward to check that $Q \left(\prod_{i = 1}^d T(\sigma_i, \tilde{B}_i)\right) Q^{-1} = \prod_{i = 1}^d T(\sigma_i, B_i)$. We have thus proven that $g$ is conjugated to $\prod_{i = 1}^d T(\sigma_i, B_i)$ in $G^\Q(\Gamma)$ with $B_i \in \GL(n_i, \Z)$ semi-simple.
\end{proof}

For any cycle $\sigma = (\lambda_1\ldots \lambda_k) \in S(\Lambda)$ consisting of coherent components of size $n$ and any $n$-tuple $\alpha = (\alpha_1, \ldots, \alpha_n) \in (K^\ast)^n$ we will use the shortened notation
\[T(\sigma, \alpha) = T(\sigma, \diag(\alpha_1, \ldots, \alpha_n)).\]
Let $W$ denote the vector subspace of $\Sp_K(V)$ spanned by the vertices in $\lambda_1 \cup \ldots \cup \lambda_k$. For any non-zero complex number $z = re^{i\theta}$ with $r > 0$ and $0 \leq \theta < 2\pi$, any positive integer $s > 0$ and any $t \in \E{s}$, we write the $t$-th $s$-root of $z$ as
\[R_{st}(z) := \sqrt[s]{r} e^{i \frac{\theta + (t-1)2\pi}{s}}.\]
The eigenvalues and corresponding eigenvectors of $T(\sigma, \alpha)|_W$ are then given by
\begin{equation}
	\label{eq:eigenvectOfT}
	R_{kj}(\alpha_i) \quad \text{and} \quad w(\sigma, \alpha)_{ij} := \sum_{m = 1}^{k} \frac{v_{\lambda_m i}}{R_{kj}(\alpha_i)^{m}},
\end{equation}
respectively, with $i \in \E{n}$ and $j \in \E{k}$.

\begin{lemma}
	\label{lem:ReducedFormOfInducedAutos}
	Let $\varphi$ be an automorphism of $A(\Gamma, c)$ and $\overline{\varphi}$ the induced morphism on $L^\C (\Gamma, c)$. There exists an automorphism $\psi \in \Aut(\overline{\Gamma})$ with a disjoint cycle decomposition $\psi = \sigma_1 \circ \ldots \circ \sigma_d$ (where 1-cycles are not omitted) and for each cycle $\sigma_i$ consisting of coherent components of size $n_i$, an element $\alpha_i = (\alpha_{i1}, \ldots, \alpha_{in_i}) \in (\C^\ast)^{n_i}$ with $\prod_{j = 1}^{n_i} \alpha_{ij} = \pm 1$ such that the Lie algebra automorphisms $\pi_c^{-1}\left( \prod_{i = 1}^d T(\sigma_i, \alpha_i) \right)$ and $\overline{\varphi}$ have the same characteristic polynomial.
\end{lemma}
\begin{proof}
	 First we will work over the field $\Q$. With abuse of notation we let $\overline{\varphi}$ also denote the induced automorphism on $L^\Q (\Gamma, c)$. Let $\overline{\varphi} = \overline{\varphi}_s \overline{\varphi}_u$ be the Jordan-Chevalley decomposition of $\overline{\varphi}$ into its semi-simple and unipotent part. Note that $\overline{\varphi}_s$ has the same characteristic polynomial as $\overline{\varphi}$. Since $\Aut_g(L^\Q (\Gamma, c))$ is a linear algebraic group over a perfect field, we know that $\overline{\varphi_s}$ still lies in $\Aut_g(L^\Q (\Gamma, c))$. Since $\pi_c$ is a linear algebraic group isomorphism to $G^\Q(\Gamma)$, we have that $\pi_c(\overline{\varphi}_s) \in G^\Q (\Gamma)$ is the semi-simple part of $\pi_c(\overline{\varphi})$. By Remark \ref{rem:charPolyIntegral}, we know that $\pi_c(\overline{\varphi})$ is integer-like which implies that also $\pi_c(\overline{\varphi}_s)$ is integer-like. Applying Lemma \ref{lem:reducedFormOverQ} to $\pi_c(\overline{\varphi}_s)$, we get an element $h \in G^\Q(\Gamma)$, an automorphism $\psi \in \Aut(\overline{\Gamma})$ with disjoint cycle decomposition $\psi = \sigma_1 \circ \ldots \circ \sigma_d$ and semi-simple matrices $B_i \in \GL(n_i, \Z)$ such that $h\pi_c(\overline{\varphi}_s)h^{-1} = \prod_{i = 1}^d T(\sigma_i, B_i)$. Since each $B_i$ is semi-simple, there exists a matrix $Q_i \in \GL(n_i, \C)$ and a tuple $\alpha_i = (\alpha_{i1}, \ldots, \alpha_{in_i}) \in (\C^\ast)^{n_i}$, such that $Q_i B_i Q_i^{-1} = \diag(\alpha_{i1}, \ldots, \alpha_{in_i})$. Since $\det(B_i) = \pm 1$, it follows that $\prod_{j = 1}^{n_i} \alpha_{ij} = \pm 1$.
\end{proof}

As a consequence of Lemma \ref{lem:reducedFormOverQ}, we can show that the characteristic polynomial of any graded integer-like automorphism of $L^\Q(\Gamma, c)$ can be realised as the characteristic polynomial of an automorphism of the group $A(\Gamma, c)$. First we show some lemma's to construct this automorphism on $A(\Gamma, c)$.

Let $\Gamma = (V, E)$ be a graph and $\Lambda$ the set of coherent components. For any $\lambda \in \Lambda$, we write $\GL(\Z, \lambda)$ for the subgroup of $\GL(\Sp_\C(\lambda))$ of linear maps that are represented by a matrix in $\GL(|\lambda|, \Z)$ with respect to the basis $\lambda$. As before, we identify $\GL(\Z, \lambda)$ with its corresponding subgroup in $\GL(\Sp_K(V))$.
\begin{lemma}
	\label{lem:blockDiagAutomorphisms}
	Let $\Gamma = (V, E)$ be a graph with set of coherent components $\Lambda$ and $K$ a subfield of $\C$. For any $B \in \left(\prod_{\lambda \in \Lambda} \GL(\Z, \lambda) \right) \leq \GL(\Sp_K(V))$, there exists an automorphism $\varphi \in \Aut(A(\Gamma, c))$ such that the induced automorphism $\overline{\varphi} \in \Aut_g(L^K (\Gamma, c))$ satisfies $\pi_c(\overline{\varphi}) = B$.
\end{lemma}
\begin{proof}
	Recall the ordering on the vertices inside each coherent component $\lambda = \{v_{\lambda 1}, \ldots, v_{\lambda |\lambda|}\}$. We write $B = \prod_{\lambda \in \Lambda} B_\lambda$ with $B_\lambda \in \GL(\Z, \lambda)$. For any $\lambda \in \Lambda$, let $(b^\lambda_{ij})_{1 \leq i,j \leq |\lambda|}$ denote the matrices which represent the maps $B_\lambda$ with respect to the basis $v_{\lambda 1}, \ldots, v_{\lambda |\lambda|}$. Let $F(V)$ be the free group on $V$. By the universal property of free groups, there exists a unique endomorphism $\varphi:F(V) \to F(V)$ which satisfies
	\[ \varphi(v_{\lambda i}) = \prod_{j = 1}^{|\lambda|} v_{\lambda j}^{b_{ji}^{\lambda}} \]
	for any $\lambda \in \Lambda$ and $i \in \E{|\lambda|}$. Since $A(\Gamma)$ is a quotient of $F(V)$, we can compose $\varphi$ with the quotient map to get a morphism $\tilde{\varphi}:F(V) \to A(\Gamma)$. Now take any two distinct vertices $v_{\lambda i}, v_{\mu i'} \in V$ such that $\{v_{\lambda i}, v_{\mu i'} \} \notin E$. By remark \ref{rem:quotientgraph} it follows that for all $1 \leq j \leq |\lambda|, \, 1 \leq j' \leq |\mu|$ also $\{v_{\lambda j}, v_{\mu j'} \} \notin E$. Thus $v_{\lambda j}$ commutes with $v_{\mu j'}$ in $A(\Gamma)$ for any $1 \leq j \leq |\lambda|, \, 1 \leq j' \leq |\mu|$. Using this we find that
	\begin{equation*}
		\tilde{\varphi}([v_{\lambda i}, v_{\mu i'}]) = [\tilde{\varphi}(v_{\lambda i}), \tilde{\varphi}(v_{\mu i'})] = \left[ \prod_{j = 1}^{|\lambda|} v_{\lambda j}^{b_{ji}^{\lambda}} , \prod_{j' = 1}^{|\mu|} v_{\mu j'}^{b_{j'i'}^{\mu}}  \right] = 1.
	\end{equation*}
	Since this holds for arbitrary distinct vertices $v_{\lambda i}, v_{\mu i'} \in V$, we find that $\tilde{\varphi}$ induces an endomorphism on $A(\Gamma)$ and thus also an endomorphism on $A(\Gamma, c)$ which we will call $\varphi$. Clearly, the endomorphism induced by $\varphi$ on the abelianization of $A(\Gamma, c)$ is invertible. Since $A(\Gamma, c)$ is nilpotent, this implies that $\varphi$ is itself an automorphism of $A(\Gamma, c)$. It is straightforward to verify that indeed $\pi_c(\overline{\varphi}) = B$.
\end{proof}

We can generalize the above lemma to the following:

\begin{lemma}
	\label{lem:blockDiagPermAutomorphisms}
	Let $\Gamma = (V, E)$ be a graph and $K$ a subfield of $\C$. Let $\psi \in \Aut(\overline{\Gamma})$ be a quotient graph automorphism with disjoint cycle decomposition $\psi = \sigma_1 \circ \ldots \circ \sigma_d$ (where 1-cycles are not omitted) and for each cycle $\sigma_i$ consisting of coherent components of size $n_i$, take any matrix $B_i$ in $\GL(n_i, \Z)$. There exists an automorphism $\varphi \in \Aut(A(\Gamma, c))$ such that the induced automorphism $\overline{\varphi} \in \Aut(L^K(\Gamma, c))$ satisfies $\pi_c(\overline{\varphi}) = \prod_{i = 1}^{d} T(\sigma_i, B_i)$.
\end{lemma}
\begin{proof}
	Note that any graph automorphism $\theta \in \Aut(\Gamma)$ naturally induces a group automorphism on $A(\Gamma)$ and thus also on $A(\Gamma, c)$. They are part of the so called elementary Nielsen automorphisms in the partially commutative setting, see \cite{serv89-1}. Now take $\theta = r(\psi)$, i.e. $\theta(v_{\lambda i}) = v_{\psi(\lambda) i}$ for any $\lambda \in \Lambda$ and $i \in \E{|\lambda|}$ and let $\varphi_1$ be the automorphism on $A(\Gamma, c)$ that it induces. It is not hard to verify that
	\[B := \pi_c(\overline{\varphi_1})^{-1} \cdot \prod_{i = 1}^{d} T(\sigma_i, B_i)\]
	is an element of $\prod_{\lambda \in \Lambda} \GL(\lambda, \Z)$. Thus by Lemma \ref{lem:blockDiagAutomorphisms}, there exists an automorphism $\varphi_2$ of $A(\Gamma, c)$ such that $\pi_c(\overline{\varphi_2}) = B$. From this it follows that the automorphism $\varphi = \varphi_1 \circ \varphi_2$ satisfies \[\pi_c(\overline{\varphi}) = \pi_c(\overline{\varphi_1}) \circ \pi_c(\overline{\varphi_2}) = \pi_c(\overline{\varphi_1}) \circ B = \prod_{i = 1}^{d} T(\sigma_i, B_i).\]
	This completes the proof.
\end{proof}

We are now ready to prove the corollary which can be seen as a sort of converse to what is stated in Remark \ref{rem:charPolyIntegral}.

\begin{cor}
	Let $\phi$ be any graded integer-like automorphism of $L^\Q(\Gamma, c)$. Then there exists an automorphism $\varphi$ of $A(\Gamma, c)$ such that its induced automorphism $\overline{\varphi}$ on $L^\Q(\Gamma, c)$ has the same characteristic polynomial as $\phi$.
\end{cor}
\begin{proof}
	Let $g$ be the semi-simple part of $\pi_c(\phi)$. Applying Lemma \ref{lem:reducedFormOverQ} to $g$, we find that $g$ is conjugated in $G^\Q(\Gamma)$ to $\prod_{i = 1}^d T(\sigma_i, B_i)$ where $\psi \in \Aut(\overline{\Gamma})$ with disjoint cycle decomposition $\psi = \sigma_1 \circ \ldots \circ \sigma_d$ and $B_i \in \GL(n_i, \Z)$. By Lemma \ref{lem:blockDiagPermAutomorphisms} there exists an automorphism $\varphi$ of $A(\Gamma, c)$ such that $\pi_c(\overline{\varphi}) = \prod_{i = 1}^d T(\sigma_i, B_i)$. We thus see that the semi-simple part of $\pi_c(\phi)$ is conjugated to $\pi_c(\overline{\varphi})$, but since $\pi_c: \Aut_g(L^\Q(\Gamma, c)) \to \GL(\Sp_\Q(V))$ is a linear algebraic group isomorphism the same must hold for $\phi$ and $\varphi$, i.e. the semi-simple part of $\phi$ is conjugated to $\varphi$. This shows that $\phi$ and $\varphi$ have the same characteristic polynomial and completes the proof.
\end{proof}

\section{Non-zero Lie brackets in $L^\C(\Gamma, c)$}
\label{sec:non-zeroLieBrackets}
In this section we prove a result that, for a given subset $W \subset \Sp_\C(V)$ satisfying certain assumptions, allows one to construct folded Lie brackets of elements in $W$ which are non-zero in $L^\C(\Gamma, c)$.

\begin{definition}
	The \textit{bracket words} on a finite set $S$ are defined inductively as follows:
	\begin{enumerate}[label = (\roman*)]
		\item The bracket words of length 1 are the elements of $S$.
		\item For any positive integer $k$, the bracket words of length $k$ are the expressions $[b_1, b_2]$ where $b_1$, $b_2$ are bracket words of lengths $k_1$, $k_2$, respectively such that $k_1 > 0$, $k_2 > 0$ and $k_1 + k_2 = k$.
	\end{enumerate}
	We write $\BW(S)$ for the set of all bracket words on $S$.
\end{definition}

The \textit{weight} of a bracket word $b \in \BW(S)$ is the unique map $e_b:S \to \N$ such that if $b$ is of length one, $e_b(s) = 1$ if $b = s$ and $e_b(s) = 0$ if $b \neq s$ and if $b$ is of length $k > 1$ with $b = [b_1, b_2]$, then $e_b = e_{b_1} + e_{b_2}$. Intuitively, the weight of $b$ counts for each element $s \in S$ how many times it occurs as a letter in the bracket word $b$. For any subset $S \subset L^\C (\Gamma, c)$ one can define a map
\[\phi_S^c:\BW(S) \to L^\C (\Gamma, c)\]
inductively by $\phi_S^c(b) = b$ if $b$ is a bracket word of length $1$ (and thus an element of $S \subset L^\C (\Gamma, c)$) and $\phi_S^c([b_1, b_2]) = [\phi_S^c(b_1), \phi_S^c(b_2)]$ for $b_1, b_2 \in \BW(S)$. For any weight $e:S \to \N$, the support of $e$ is defined by $\supp(e) = \{ s \in S \mid e(s) \neq 0 \}$. Define the set of weights on $V$ with sum at most $c$ as
\begin{equation}
	\label{eq:weightsOfSumAtMostc}
	\mathcal{E}(V, c) = \left\{ e:V \to \N \: \middle| \: \sum_{v \in V} e(v) \leq c \right\}.
\end{equation}

\begin{definition}
	Let $\Gamma = (V, E)$ be a graph. We say a subset $A \subset V$ is \textit{connected} in $\Gamma$ if for any two distinct vertices $v,w \in A$ there exist vertices $v_1,  v_2, \ldots, v_{n-1}, v_n \in A$ with $v_1 = v$ and $v_n = w$ such that $\{v_i, v_{i+1}\} \in E$ for any $i \in \E{n-1}$.
\end{definition}

The next Lemma follows from combining Lemma 3.4. and Theorem 3.6. from \cite{dw23-1}.

\begin{lemma}
	\label{lem:existenceNonZeroBWWithGivenWeight}
	Let $\Gamma = (V, E)$ be a graph $c > 1$ an integer and $e \in \mathcal{E}(V, c)$ a weight. If $|\supp(e)| \geq 2$ and $\supp(e)$ is connected, then there exists a bracket word $b \in \BW(V)$ of weight $e$ such that $\phi_V^c(b)$ is non-zero.
\end{lemma}

 As a vector space, $L^\C (\Gamma, c)$ can be written as a direct sum
\begin{equation}
	\label{eq:vectorSpaceDecompLieAlgGraph}
	L^\C (\Gamma, c) = \bigoplus_{e \in \mathcal{E}(V, c)} \Sp_\C\left( \left\{ \phi_V^c(b) \mid b \in \BW(V) \text{ of weight } e \right\} \right).
\end{equation}
Using this decomposition, we can prove the following generalization of Lemma \ref{lem:existenceNonZeroBWWithGivenWeight}.

\begin{lemma}
	\label{lem:existenceNonZeroBWWithGivenWeightGeneralization}
	Let $\Gamma = (V, E)$ be a graph and $A \subset V$ a connected subset with $|A| \geq 2$. Let $W \subset \Sp_\C(V) \subset L^\C (\Gamma, c)$ be a finite set of vectors such that there exists a surjective map $\kappa:W \to A$ which satisfies
	\begin{equation*}
		\label{eq:conditionsOnKappa}
		\forall w \in W: \: w \in \Sp_\C(\kappa(w) \cup (V \setminus A)) \quad \text{and} \quad w \notin \Sp_\C(V \setminus A).
	\end{equation*}
	For any weight $e \in \mathcal{E}(W, c)$ with $\supp(e) = W$, there exists a bracket word $b \in \BW(W)$ of weight $e$ such that $\phi_W^c(b)$ is non-zero.
\end{lemma}
\begin{proof}
	Note that the map $\kappa:W \to A$ induces a map on bracket words $\overline{\kappa}:\BW(W) \to \BW(V)$ defined inductively by $\overline{\kappa}(w) = \kappa(w)$ if $w$ is a bracket word of length 1 (and thus an element of $W$), and $\overline{\kappa}([b_1, b_2]) = [\overline{\kappa}(b_1), \overline{\kappa}(b_2)]$ for any $b_1, b_2 \in \BW(W)$. From the weight $e:W \to \N$ we define a new weight $\tilde{e} \in \mathcal{E}(V, c)$ by
	\[ \tilde{e}:V \to \N: a \mapsto \begin{cases}
		\sum_{w \in \kappa^{-1}(a)} e(w) &\text{ if } a \in A\\
		0 & \text{ else.}
	\end{cases}  \]
	From the fact that $\supp(e) = W$ and that $\kappa$ is surjective, it follows that $\supp(\tilde{e}) = A$. Using Lemma \ref{lem:existenceNonZeroBWWithGivenWeight}, there exists a bracket word $\tilde{b} \in \BW(V)$ with weight $\tilde{e}$, such that $\phi_V^c(\tilde{b}) \in L^\C (\Gamma, c)$ is non-zero. By the way we constructed $\tilde{e}$ from $e$, it is clear that one can choose a bracket word $b \in \BW(W)$ such that $\overline{\kappa}(b) = \tilde{b}$ and such that $b$ has weight $e$. 
	
	For each $w \in W$, let $f_w:V \to \C$ denote the unique function such that $w = \sum_{v \in V} f_w(v) v$. From the assumption on $\kappa$, it is clear that for any $w \in W$ we have
	\[ w = f_w(\kappa(w)) \kappa(w) + \sum_{v \in V \setminus A} f_w(v) v \]
	with $f_w(\kappa(w)) \neq 0$. Using the linearity of the Lie bracket in $L^\C (\Gamma, c)$, we can thus rewrite the element $\phi_W(b)$ as a sum
	\[ \phi_W^c(b) = \left( \prod_{w \in W} f_w(\kappa(w))^{e(w)} \right) \cdot \phi_V^c(\tilde{b}) + \sum_{d \in D} a_d \cdot \phi_V^c(d) \]
	for some coefficients $a_d \in \C$ and a subset $D \subset \BW(V)$ of bracket words in $V$ each having a weight with support not contained in $A$. In particular, each $d \in D$ has weight not equal to $\tilde{e}$, the weight of $\tilde{b}$. Using the vector space direct sum decomposition of $L^\C (\Gamma, c)$ as given in (\ref{eq:vectorSpaceDecompLieAlgGraph}) and the fact that $f_w(\kappa(w)) \neq 0$ for all $w \in W$ and $\phi_V^c(\tilde{b}) \neq 0$, it follows that $\phi_{W}^c(b)$ is non-zero.
\end{proof}

In section \ref{sec:upperBound} we will apply the above lemma to a subset $W$ of eigenvectors of the automorphism $\pi_c^{-1}(T(\sigma_1, \alpha_1) \cdot \ldots \cdot T(\sigma_d, \alpha_d))$ which are given by the vectors $w(\sigma, \alpha)_{ij}$ from equation (\ref{eq:eigenvectOfT}). This will give us a tool for constructing an eigenvector with eigenvalue 1.

\section{Transposition-free graphs}
\label{sec:transposition-freeGraphs}

In this section we prove Theorem \ref{thm:transpositionFreeIndex2}. This is an improvement to a result in \cite{send21-1} where it is proven (although not stated as a theorem) that the $R_\infty$-nilpotency index of the RAAG's associated to the smaller class of non-empty transvection-free graphs is equal to 3. Note that the author uses the opposite convention to define RAAG's, namely that adjacent vertices commute.

Let $\Gamma = (V, E)$ be a non-empty transposition-free graph. This is equivalent with saying that all its coherent components are singletons and that $\Gamma$ has at least 2 vertices. In this case, the graph $\Gamma$ and its quotient graph $\overline{\Gamma}$ are essentially the same and we get a natural identification $\Aut(\Gamma) \cong \Aut(\overline{\Gamma})$, which we use in what follows.

Applying Lemma \ref{lem:ReducedFormOfInducedAutos} to $\Gamma$ for any $c > 1$, we find that for any automorphism $\varphi$ of $A(\Gamma, c)$, there exists a graph automorphism $\psi \in \Aut(\Gamma) \cong \Aut(\overline{\Gamma})$ with disjoint cycle decomposition
\[ \psi = \sigma_1 \circ \ldots \circ \sigma_d, \quad \quad \quad \sigma_i = (v_{i1} v_{i2} \ldots v_{ik_i})\]
where $v_{ij} \in V$ and there exist $\alpha_1, \ldots, \alpha_d \in \{-1, 1\}$ such that the induced automorphism $\overline{\varphi} \in \Aut_g(L^\C (\Gamma))$ has the same characteristic polynomial as the automorphism $\pi_c^{-1}\left( \prod_{i = 1}^d T(\sigma_i, \alpha_i) \right)$ where the maps $T(\sigma_i, \alpha_i)$ have matrix representation $(\pm 1)$ if $k_i = 1$ and
\[ \begin{pmatrix}
	0 & & & \pm 1 \\
	1 & 0 & &\\
	& \ddots & \ddots & \\
	& & 1 & 0
\end{pmatrix}\]
if $k_i > 1$ with respect to the basis $v_{i1}, v_{i2}, \ldots, v_{ik_i}$. Combining this with Lemma \ref{lem:reidemeisterNumberFinGenNilpGroup}, we find that in order to prove Theorem \ref{thm:transpositionFreeIndex2}, it suffices to show that $\pi_2^{-1}\left( \prod_{i = 1}^d T(\sigma_i, \alpha_i) \right)$ has an eigenvector with eigenvalue 1. In the following cases, we can always write down such an eigenvector.

\begin{enumerate}[label = C\arabic*.]
	\item \label{item:transFreeCase1}There exists an index $i \in \{1, \ldots, d\}$ such that $\alpha_i = 1$. Then it is clear that $v_{i1} + \ldots + v_{ik_i} \in L_1(\Gamma, 2)$ is an eigenvector with eigenvalue 1.
	\item \label{item:transFreeCase2}There exists indices $i \in \E{d},\, s,t \in \E{k_i}$ such that $\alpha_i = -1$ and $\{v_{is}, v_{it}\} \in E$. Depending on whether $k_i$ divides $2(s-t)$ or not, the vector
	\[ \sum_{m = 1}^{k_i/2} \left[\psi^m(v_{is}), \psi^m(v_{it})\right] \quad \text{ or } \quad \sum_{m = 1}^{k_i} \left[\psi^m(v_{is}), \psi^m(v_{it})\right], \]
	respectively, will be an eigenvector with eigenvalue 1 in $L_2^\C(\Gamma, 2)$. The case distinction stems from the fact that an eigenvector needs to be non-zero.
	\item \label{item:transFreeCase3}There exist indices $i,j \in \E{d}, s \in \E{k_i}, t \in \E{k_j}$ such that $\alpha_i = \alpha_j = -1$, $\{v_{is}, v_{jt}\} \in E$ and $\lcm(k_i, k_j)\left( \frac{1}{k_i} + \frac{1}{k_j} \right)$ is even. Then the vector
	\[ \sum_{m = 1}^{\lcm(k_i, k_j)} [\psi^m(v_{is}), \psi^m(v_{jt})] \]
	will be an eigenvector with eigenvalue 1 on $L_2^\C(\Gamma, 2)$.
\end{enumerate}

\noindent In \cite{send21-1}, the author considers some other cases where an eigenvector with eigenvalue one is constructed on $L_3(\Gamma, c)$. In what follows, we will exploit the fact that the graph is assumed to be non-empty transposition-free to show that at least one of the cases \ref{item:transFreeCase1}, \ref{item:transFreeCase2} or \ref{item:transFreeCase3} is always true. This is formulated in Proposition \ref{prop:edgeExistenceBetweenCycles} and thus proves Theorem \ref{thm:transpositionFreeIndex2} from the introduction. 

First, we prove the following lemma which tells us something about the neighbours of vertices lying in the same cycle of an automorphism $\psi \in \Aut(\Gamma)$. Recall the definition of the open neighbourhood $\Omega'(v)$ of a vertex $v \in V$ from \eqref{eq:neighbourhoodVertex}.

\begin{lemma}
	\label{lem:neighboursBetweenCycles}
	Let $\Gamma$ be a graph and $\psi \in \Aut(\Gamma)$ with disjoint cycle decomposition $\psi = \sigma_1 \circ \sigma_2 \circ \ldots \circ \sigma_d$, where $\sigma_i = (v_{i1} \: v_{i2} \ldots v_{ik_i})$. For any $i,j \in \E{d}, \, s,t \in \E{k_i}$ it holds that if $\gcd(k_i, k_j) \mid (s-t)$, then $\Omega'(v_{is}) \cap V_j = \Omega'(v_{it}) \cap V_j$ where $V_j$ denotes the set of vertices in the cycle $\sigma_j$.
\end{lemma}
\begin{proof}
	Assume $\gcd(k_i, k_j) \mid (s-t)$. It follows by the theorem of Bachet-Bézout, that there exist integers $\alpha, \beta \in \Z$ such that $\alpha k_i + t-s = \beta k_j$. Take an arbitrary vertex $w \in \Omega'(v_{is}) \cap V_j$. By definition we have that $\{v_{is}, w\} \in E$. Since $\psi$ is a graph automorphism we have as well that $\{ \psi^{\alpha k_i + t-s}(v_{is}), \psi^{\beta k_j}(w) \} \in E$. Using the cycle decomposition of $\psi$, we get the equality $\{ \psi^{\alpha k_i + t-s}(v_{is}), \psi^{\beta k_j}(w) \} = \{ v_{it}, w \}$. This shows that $w \in \Omega'(v_{it})$ and thus that $\Omega'(v_{is}) \cap V_j \subseteq \Omega'(v_{it}) \cap V_j$. This finishes the proof since the other inclusion is analogous .
\end{proof}

Using the above lemma we can prove our main structural result on a graph and the cycle decomposition of one of its automorphisms.

\begin{prop}
	\label{prop:edgeExistenceBetweenCycles}
	Let $\Gamma$ be a non-empty transposition-free graph and $\psi \in \Aut(\Gamma)$ with disjoint cycle decomposition $\psi = \sigma_1 \circ \sigma_2 \circ \ldots \circ \sigma_d$, where $\sigma_i = (v_{i1} \: v_{i2} \ldots v_{ik_i})$. Then at least one of the following is true:
	\begin{enumerate}[label = (\roman*)]
		\item \label{item:edgeExistenceBetweenCycles1} $\exists \; i \in \E{d}, \, s, t \in \E{k_i}: \: \{v_{is}, v_{it}\} \in E$.
		\item \label{item:edgeExistenceBetweenCycles2} $\exists \; i,j \in \E{d}, \, s \in \E{k_i}, \, t \in \E{k_j}:\: i \neq j, \, \{v_{is}, v_{jt}\} \in E$ and $\lcm(k_i, k_j)\cdot \left( \frac{1}{k_i} + \frac{1}{k_j} \right)$ is even.
	\end{enumerate}
\end{prop}
\begin{proof}
	We will prove this statement by contradiction. Assume both \ref{item:edgeExistenceBetweenCycles1} and \ref{item:edgeExistenceBetweenCycles2} do not hold. Let us define for any $i \in \E{d}$ the set of indices 
	\[\Psi(i) = \{ j \in \E{d} \setminus \{i\} \, \mid \exists s \in \E{k_i}, \, t \in \E{k_j},\, \{v_{is}, v_{jt}\} \in E \}.\]
	From the assumption that \ref{item:edgeExistenceBetweenCycles2} does not hold, it follows immediately that for any $j \in \Psi(i)$ the integer $\lcm(k_i, k_j) \cdot \left( \frac{1}{k_i} + \frac{1}{k_j}\right)$ is odd. We can write each cycle length $k_i$ uniquely as $k_i = 2^{e_i} m_i$ with $e_i \in \N$, $m_i \in \N\setminus \{0\}$ and $2 \nmid m_i$. Now suppose that $j \in \Psi(i)$ and that $e_i = e_j$. We then have that
	\begin{align*}
	\lcm(k_i, k_j) \cdot \left( \frac{1}{k_i} + \frac{1}{k_j}\right) &= \frac{\lcm(k_i, k_j)}{k_i} + \frac{\lcm(k_i, k_j)}{k_j}\\
	&= \frac{2^{e_i} \lcm(m_i, m_j)}{2^{e_i} m_i} + \frac{2^{e_i} \lcm(m_i, m_j)}{2^{e_i} m_j}\\
	&= \frac{\lcm(m_i, m_j)}{m_i} + \frac{\lcm(m_i, m_j)}{m_j}
	\end{align*}
	is even, since both summands $\lcm(m_i, m_j)/m_i$ and $\lcm(m_i, m_j)/m_j$ are odd. Clearly this contradicts the assumption and thus we must conclude that for any $j \in \Psi(i)$, $e_i \neq e_j$.
	
	Take any $i \in \E{d}$ for which $\Psi(i)$ is non-empty and suppose that $e_j < e_i$ for any $j \in \Psi(i)$. Then we get for any $j \in \Psi(i)$ that $\gcd(k_i, k_j) = 2^{e_j} \gcd(m_i, m_j)$ which divides $2^{e_i - 1} m_i = k_i/2$. By Lemma \ref{lem:neighboursBetweenCycles} we thus have for any $j \in \Psi(i)$ that $\Omega'(v_{i1}) \cap V_j = \Omega'(v_{i r_i}) \cap V_j$ where we write $r_i := 1 + k_i/2$. Note that by the definition of $\Psi(i)$, we also have that $\Omega'(v_{i1}) \cap V_j = \emptyset = \Omega'(v_{i r_i}) \cap V_j$ for any $j \in \E{d} \setminus (\Psi(i) \cup \{i\})$. At last, we also find that $\Omega'(v_{i1}) \cap V_i = \emptyset = \Omega'(v_{i r_i}) \cap V_i$ as a consequence of the assumption that condition \ref{item:edgeExistenceBetweenCycles1} does not hold. Since $\{V_i \mid 1 \leq i \leq d\}$ is a partition of $V$, we can combine these equalities to get $\Omega'(v_{i1}) = \Omega'(v_{ir_i})$ which contradicts the fact that $\Gamma$ is transposition-free. Therefore we must have that if $\Psi(i)$ is non-empty, then there exists a $j \in \Psi(i)$ such that $e_j > e_i$.
	
	Since $\Gamma$ is not the empty graph and since we assume that \ref{item:edgeExistenceBetweenCycles1} does not hold, we must have that there exists an $i_0 \in \E{d}$ such that $\Psi(i_0)$ is non-empty. By the above, we know that there is an index $i_1 \in \Psi(i_0)$ such that $e_{i_1} > e_{i_0}$. Note that $\Psi(i_1)$ is also non-empty since $i_0 \in \Psi(i_1)$. We can repeat the argument to find an index $i_2 \in \Psi(i_1)$ such that $e_{i_2} > e_{i_1}$. If we keep repeating this process, we clearly get a contradiction since the set $\{e_i \mid i \in \E{d}\}$ is finite.
\end{proof}

The above proves that if a graph $\Gamma$ is transposition-free and has more than one vertex, then $A(\Gamma)$ has $R_\infty$-nilpotency index 2. The converse does not hold. Indeed, consider the graph on four vertices as drawn below.
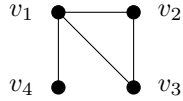
\begin{figure}[H]
	\centering
	\begin{tikzpicture}
		\filldraw [black] (0.5,0.5) circle (2.5pt) node[right = 2mm] {$v_2$};
		\filldraw [black] (0.5,-0.5) circle (2.5pt) node[right = 2mm] {$v_3$};
		\filldraw [black] (-0.5,-0.5) circle (2.5pt) node[left = 2mm] {$v_4$};
		\filldraw [black] (-0.5,0.5) circle (2.5pt) node[left = 2mm] {$v_1$};
	
		\draw (-0.5, 0.5) -- (0.5,0.5);
		\draw (0.5, 0.5) -- (0.5,-0.5);
		\draw (-0.5, 0.5) -- (0.5,-0.5);
		\draw (-0.5, 0.5) -- (-0.5,-0.5);
	\end{tikzpicture}
\end{figure}
\noindent This graph is clearly not transposition-free, since the transposition of $v_2$ with $v_3$ is a graph automorphism, but $A(\Gamma, 2)$ does have the $R_\infty$-property. This is not so hard to prove using Lemma \ref{lem:ReducedFormOfInducedAutos}. For an actual proof, we also refer to \cite{dl23-1} (where we note that the author uses the opposite convention to define a RAAG). This raises the following question.

\begin{question}
	For which graphs $\Gamma$ does $A(\Gamma)$ have $R_\infty$-nilpotency index equal to 2?
\end{question}
\section{Bounds on the $R_\infty$-nilpotency index of a RAAG}
In this section we prove Theorem \ref{thm:allRAAGSRinfty}. We divide the proof into two parts. The first part deals with the lower bound $\xi(\Gamma)$ and the second with the upper bound $\Xi(\Gamma)$ as defined by equations (\ref{eq:lowerBound}) and (\ref{eq:upperBound}), respectively. 

\subsection{Lower bound}
\label{sec:lowerBound}

First, we show the existence of finitely many polynomials of arbitrary degree for which certain products of their roots are not equal to one. These polynomials will then be used to construct automorphisms on $A(\Gamma, c)$ with finite reidemeister number. A useful tool to construct these polynomials are the so called Pisot units. A \textit{Pisot number} is a real algebraic integer greater than one such that all its conjugates have absolute value stricly less than one. A \textit{Pisot unit} is a Pisot number which is also an algebraic unit, i.e. an algebraic integer for which the constant term of its minimal polynomial over $\Q$ is equal to $\pm 1$. A proof of the following fact can be found in Lemma 2.7. of \cite{dg14-1} and in Proposition 3.6.(3) of \cite{payn09-1}.

\begin{lemma}
	\label{lem:pisotUnit}
	Let $\alpha_1$ be a Pisot unit with conjugates $\alpha_2, \ldots, \alpha_d$. For any $e = (e_1, \ldots, e_d) \in \Z^d$ it holds that if $\prod_{i = 1}^{d} \alpha_i^{e_i} = 1$ then $e_1 = \ldots = e_d$.
\end{lemma}

We can now prove the following existence result.

\begin{lemma}
	\label{lem:polynomialsWithoutNonTrivialUnitProductsOfEigenvalues}
	Let $n > 0$, $c > 1$ and $d_1, \ldots, d_n > 0$ be any positive integers. There exist monic irreducible polynomials $p_1(X), \ldots, p_n(X) \in \Z[X]$ of degree $d_1, \ldots, d_n$, respectively, such that if $\alpha_{i 1}, \ldots, \alpha_{i d_i} \in \C$ denote the the zeros of $p_i(X)$, the following are true:
	\begin{enumerate}[label = (\roman*)]
		\item $\forall\,  i \in \E{n}:\: \alpha_{i1} \cdot \ldots \cdot \alpha_{i d_i} = - 1$,
		\item $\forall e = (e_{11}, \ldots, e_{1 d_1}, e_{21}, \ldots , e_{2d_2}, \ldots, e_{n1}, \ldots , e_{nd_n}) \in \left(\N\right)^{d_1 + \ldots + d_n}$ with $\sum_{i = 1}^n \sum_{j = 1}^{d_i} e_{ij} \leq c$:
		\[\prod_{i = 1}^n \prod_{j = 1}^{d_i} \alpha_{ij}^{e_{ij}} = 1 \quad \Rightarrow \quad \left(\,\forall i \in \E{n}: \: e_{i1} = \ldots = e_{id_i}\,\right).\]
	\end{enumerate}
\end{lemma}
\begin{proof}
	From Lemma 2.7. in \cite{dg14-1}, it follows that for all $i \in \E{n}$, there exists a monic $\Q$-irreducible polynomial $q_i(X) \in \Z[X]$ of degree $d_i$ such that if $\beta_{i1}, \ldots, \beta_{id_i}$ denote its roots, we have $\beta_{i 1} \cdot\ldots \cdot \beta_{i d_i} = -1$ and $\beta_{i 1}$ is a Pisot number. Next, for any $e = (e_{11}, \ldots, e_{1 d_1}, e_{21}, \ldots , e_{2d_2}, \ldots, e_{n1}, \ldots , e_{nd_n}) \in \left(\N\right)^{d_1 + \ldots + d_n}$ define the map
	\[ \varphi_e:\Z^n \to \C^\ast: (k_1, \ldots , k_n) \mapsto \prod_{i = 1}^n \prod_{j = 1}^{d_i} \left(\beta_{ij}^{k_i}\right)^{e_{ij}}, \]
	which is a group morphism between the additive group $\Z^n$ and the multiplicative group $\C^\ast = \C \setminus \{0\}$. The kernel $\ker(\varphi_e)$ is a free abelian subgroup of $\Z^n$, with rank less or equal to $n$. Assume that $\rank(\ker(\varphi_e)) = n$. Take any $i \in \E{n}$. Then there exists a $k_i \in \Z$ such that the element $(0, \ldots, 0, k_i, 0, \ldots, 0)$, where the non-zero element is on the $i$-th entry, lies in $\ker(\varphi_e)$. We thus have that $\displaystyle \prod_{j = 1}^{d_i} \left(\beta_{ij}^{k_i}\right)^{e_{ij}} = 1$. Note that $\beta_{i1}^{k_i}$ is still a Pisot number with conjugates $\beta_{i2}^{k_i}, \ldots, \beta_{id_i}^{k_i}$. By Lemma \ref{lem:pisotUnit} we thus have that $e_{i1} = \ldots = e_{i d_i}$. Since $i$ was chosen arbitrarily, we get $\forall i \in \E{n}: \: e_{i1} = \ldots = e_{id_i}$. This proves that for any $e \in \Z^{d_1 + \ldots + d_n}$ with $\neg(\forall i \in \E{n}: \: e_{i1} = \ldots = e_{id_i})$ the kernel $\ker(\varphi_e)$ has rank strictly less than $n$. As a consequence
	\[ \mathcal{H} = \left\{ \ker(\varphi_e) \: \middle| \: e \in (\N)^{d_1 + \ldots + d_n}: \: \neg(\forall i \in \E{n}: \: e_{i1} = \ldots = e_{id_i}) \: \land \: \sum_{i = 1}^n \sum_{j = 1}^{d_i} e_{ij} \leq c \right\} \]
	is a finite collection of subgroups of $\Z^n$ of rank strictly less than $n$. It follows that we can find an element $m = (m_1, \ldots, m_n) \in \Z^n$ such that $m \notin \bigcup_{H \in \mathcal{H}} H$ and such that each integer $m_i$ is odd. For each $i \in \E{n}$, we then define the polynomial
	\[ p_i(X) = (X - \beta_{i 1}^{m_i}) \cdot \ldots \cdot (X - \beta_{i d_i}^{m_i}). \]
	As one can check, these polynomials satisfy all requirements of the lemma.
\end{proof}

We also need the following result which was proven in Proposition 3.7. from \cite{dw23-1} and tells us what the eigenvalues are of an automorphism $\varphi \in \Aut(L^\C(\Gamma, c))$ which is diagonal on the vertices $V$. Recall the notations $\supp(e)$ and $\mathcal{E}(V, c)$ as introduced in section \ref{sec:non-zeroLieBrackets}.

\begin{lemma}
	\label{lem:eigenvaluesVertexDiagAuto}
	Let $\Gamma$ be a graph and $c > 1$. Let $\varphi \in \Aut(L^\C(\Gamma, c))$ be a vertex-diagonal automorphism, i.e. an automorphism such that there exists a function $\Psi: V \to \C^\ast$ with $\forall v \in V: \varphi(v) = \Psi(v) v$. The set of eigenvalues of $\varphi$ is equal to
	\[ \Psi(V) \cup \left\{ \prod_{v \in V} \Psi(v)^{e(v)}  \: \middle| \: \begin{array}{l}
		e \in \mathcal{E}(V, c), \, \, |\supp(e)| \geq 2,\\
		\supp(e) \text{ is connected in } \Gamma
	\end{array} \right\}.\]  
\end{lemma}

We are now ready to prove the lower bound on the $R_\infty$-nilpotency index of $A(\Gamma)$ of which we recall the definition:
\[ 	\xi(\Gamma) = \min \left\{ |\lambda| + |\mu| \: \middle| \: \lambda, \mu \in \Lambda,\, \{\lambda , \mu\} \in \overline{E} \right\}. \]

\begin{theorem}
	Let $\Gamma$ be a non-empty graph. The group $A(\Gamma, \xi(\Gamma) - 1)$ does not have the $R_\infty$-property.
\end{theorem}
\begin{proof}
	Let $\overline{\Gamma} = (\Lambda, \overline{E}, \Phi)$ be the quotient graph of $\Gamma$ and write $c = \xi(\Gamma) -1$. Using Lemma \ref{lem:polynomialsWithoutNonTrivialUnitProductsOfEigenvalues}, we can find for any $\lambda \in \Lambda$ a monic irreducible polynomial $p_\lambda(X) \in \Z[X]$ of degree $|\lambda|$ with eigenvalues $\alpha_{\lambda 1}, \ldots, \alpha_{\lambda |\lambda|}$ such that $\alpha_{\lambda 1} \cdot \ldots \cdot \alpha_{\lambda |\lambda|} = -1$ and 
	\[\forall e_{\lambda i} \in \N \text{ with } \sum_{\lambda \in \lambda} \sum_{j = 1}^{|\lambda|} e_{\lambda j} \leq c: \quad \prod_{\lambda \in \Lambda} \prod_{j = 1}^{|\lambda|} \alpha_{\lambda j}^{e_{\lambda j}} = 1 \quad \Rightarrow \quad \left(\forall \, \lambda \in \Lambda: e_{\lambda 1} = \ldots = e_{\lambda |\lambda|}\right).\]
	Recall that we fixed an ordering of the vertices in each coherent component $\lambda = \{v_{\lambda 1}, \ldots , v_{\lambda |\lambda|}\}$. For each $\lambda \in \Lambda$, let $B_\lambda \in \GL(\Z, \lambda)$ be the linear map such that its matrix representation w.r.t. the basis $v_{\lambda 1}, \ldots , v_{\lambda |\lambda|}$ is given by the companion matrix of $p_\lambda(X)$. This gives an element $B = \prod_{\lambda \in \Lambda} B_\lambda$ in $\prod_{\lambda \in \Lambda} \GL(\Z, \lambda)$. By Lemma \ref{lem:blockDiagAutomorphisms}, there exists an automorphism $\varphi \in \Aut(A(\Gamma, c))$ such that, if $\overline{\varphi} \in \Aut_g(L^\C (\Gamma, c))$ denotes the induced automorphism on the Lie algebra, $\pi_c(\overline{\varphi}) = B$. Equivalently we have $\pi_c^{-1}(B) = \overline{\varphi}$. Note that, by Lemma \ref{lem:reidemeisterNumberFinGenNilpGroup}, if $\pi_c^{-1}(B) = \overline{\varphi}$ does not have an eigenvalue equal to 1, the Reidemeister number of $\varphi$ is not equal to $\infty$ and thus $A(\Gamma, c)$ does not have the $R_\infty$-property. Let us prove by contradiction that this is indeed the case.
	
	Assume $\pi_c^{-1}(B)$ has an eigenvalue equal to one. Note that companion matrices are diagonalizable and thus that for each $\lambda \in \Lambda$ there exists a $Q_\lambda \in \GL(\Sp_\C(\lambda))$ such that the matrix representation of $Q_\lambda B_\lambda Q_\lambda^{-1}$ w.r.t. the basis $v_{\lambda 1}, \ldots , v_{\lambda |\lambda|}$ is given by $\diag(\alpha_{\lambda 1}, \ldots , \alpha_{\lambda |\lambda|})$. It follows that 
	\[D := \pi_c^{-1}\left(\left(\prod_{\lambda \in \Lambda} Q_\lambda \right) B \left(\prod_{\lambda \in \Lambda} Q_\lambda \right)^{-1}\right)\]
	is an automorphism of $L^\C (\Gamma, c)$ which is diagonal on the vertices and which has the same eigenvalues as $\pi_c^{-1}(B)$. We thus have that $D$ has an eigenvalue equal to one. Note that this eigenvalue can not occur on $\Sp_\C(V)$ since $\alpha_{\lambda j} \neq 1$ for all $\lambda \in \Lambda$ and $j \in \E{|\lambda|}$. By Lemma \ref{lem:eigenvaluesVertexDiagAuto} above, there must therefore exist $e_{\lambda j} \in \N$ such that \[s := \sum_{\lambda \in \Lambda} \sum_{j = 1}^{|\lambda|} e_{\lambda j} \leq c,\quad \quad \prod_{\lambda \in \Lambda} \prod_{j = 1}^{|\lambda|} \alpha_{\lambda j}^{e_{\lambda j}} = 1\]
	and the set $A := \{ v_{\lambda j} \mid \lambda \in \Lambda, 1 \leq j \leq |\lambda|, e_{\lambda j} \neq 0 \}$ has size greater or equal to two and is connected in $\Gamma$. From the properties of the polynomials $p_\lambda(X)$, we have for each $\lambda \in \Lambda$ that $e_{\lambda 1} = \ldots = e_{\lambda |\lambda|}$. As a consequence $A$ is equal to a union of coherent components. This gives two cases:
	\begin{itemize}
		\item If there is only one coherent component in this union, say $A = \lambda \in \Lambda$, then it follows from $A$ being connected, that $\{\lambda\} \in \overline{E}$. From the definition of $\xi(\Gamma)$ we thus have that $\xi(\Gamma) \leq 2|\lambda|$. We must also have that $e_{\lambda 1} = \ldots = e_{\lambda |\lambda|}$ are even integers since $\alpha_{\lambda 1} \cdot \ldots \cdot \alpha_{\lambda |\lambda|} = -1$. We thus have that $2|\lambda| \leq s$ which implies $\xi(\Gamma) \leq 2|\lambda| \leq s$. This gives a contradiction with $s \leq c = \xi(\Gamma) - 1$.
		\item If, on the other hand, $A$ contains more than one coherent component, then there must exist $\lambda, \mu \subset A$ with $\lambda \neq \mu$ and $\{\lambda, \mu\} \in \overline{E}$ since $A$ is connected. We thus have that $\xi(\Gamma) \leq |\lambda| + |\mu| \leq s$, which again contradicts the fact that $s \leq c = \xi(\Gamma) - 1$. 
	\end{itemize}
 	Thus, we must conclude that $\pi_c^{-1}(B) = \overline{\varphi}$ does not have an eigenvalue equal to one.
\end{proof}

\subsection{Upper bound}
\label{sec:upperBound}
For the upper bound (see (\ref{eq:upperBound}) for the definition), one has to show that if $c = \Xi(\Gamma)$ then for each automorphism $\varphi \in \Aut(A(\Gamma, c))$, the induced automorphism $\overline{\varphi} \in \Aut_g(L^\C (\Gamma, c))$ has an eigenvalue one. We will construct such an eigenvalue 1 as a product of eigenvalues of $\overline{\varphi}$ on $L_1^\C(\Gamma, c)$. If one wants the corresponding eigenvector in $L_i^\C(\Gamma, c)$ with $i$ as small as possible, one needs to solve the following problem: given two positive integers $k,l$ and elements $\alpha, \beta \in \{-1, 1\}$, how can we obtain $1$ as a product of $k$-roots of $\alpha$ and $l$-roots of $\beta$ with the least factors as possible, where we do need to take at least one $k$-root of $\alpha$ and one $l$-root of $\beta$.

The answer, of course, depends on the integers $k,l,\alpha, \beta$, but as it turns out, one always needs either 2 or 3 factors. In case $\alpha = \beta= 1$, the problem is trivial. Since 1 is a $k$-root of $\alpha$ and 1 is an $l$-root of $\beta$, one needs only two factors: $1 \cdot 1 = 1$. The cases $\alpha = \beta = -1$ and $\alpha = -1, \beta = 1$ are less trivial and are treated in the following two lemmas.

Recall that we defined the roots of unity
\[R_{st}(z) := \sqrt[s]{r} e^{i \frac{\theta + (t-1)2\pi}{s}}\]
for any non-zero complex number $z = re^{i\theta}$ with $r > 0$ and $0 \leq \theta < 2\pi$ and any positive integers $s > 0$ and and $t \in \E{s}$.

\begin{lemma}
	\label{lem:prodRootsEqual1SameSign}
	Take any positive integers $k, l > 0$ and write $M = \lcm(k,l)$. Then the following are true:
	\begin{enumerate}[label = (\roman*)]
		\item \label{item:prodRootsEq1Case0} There exist integers $s \in \E{k}$ and $t \in \E{l}$ such that
		\[ R_{ks}(1) \cdot R_{lt}(1) = 1. \]
		\item \label{item:prodRootsEq1Case1} If $M \cdot \left( \frac{1}{k} + \frac{1}{l} \right)$ is even, then there exist integers $s \in \E{k}$ and $t \in \E{l}$ such that
		\[ R_{ks}(-1) \cdot R_{lt}(-1) = 1.\]
		\item \label{item:prodRootsEq1Case2} If $M/k$ is odd and $M/l$ is even, then there exist integers $s, r \in \E{k}$ and $t \in \E{l}$ such that
		\[ R_{ks}(-1) \cdot R_{kr}(-1) \cdot R_{lt}(-1) = 1.\]
	\end{enumerate}
\end{lemma}
\begin{proof}
	Statement \ref{item:prodRootsEq1Case0} is trivial, namely take $s = t = 1$. For \ref{item:prodRootsEq1Case1}, note that $M/k$ and $M/l$ are coprime and thus by the Theorem of Bachet-B\'ezout, there exist integers $s, t \in \Z$ such that
	\[ \frac{M}{k} + \frac{M}{l} + 2 \left( (s-1) \frac{M}{k} + (t-1) \frac{M}{l}\right) \equiv 0 \mod 2M, \]
	where we used the assumption that $M/k + M/l$ is even. Moreover, since the equation is modulo $2M$, it is clear that we can take $s, t$ such that $s \in \E{k}$ and $t \in \E{l}$. Dividing the left-hand side by $M$, we thus find that $\frac{1}{k} + \frac{1}{l} + \frac{2(s-1)}{k} + \frac{2(t-1)}{l} = \frac{2s-1}{k} + \frac{2t-1}{l}$ is an even integer. From this it follows that
	\[ R_{ks}(-1) \cdot R_{lt}(-1) = e^{i\left( \frac{\pi}{k} + (s-1)\frac{2\pi}{k} \right)} e^{i\left( \frac{\pi}{l} + (t-1)\frac{2\pi}{l} \right)} = e^{i \pi \left(\frac{2s-1}{k} + \frac{2t-1}{l} \right)} = 1.\]
	For \ref{item:prodRootsEq1Case2}, note that $M/k$ and $M/l$ are coprime and thus by the Theorem of Bachet-B\'ezout, there exist integers $s,r, t \in \Z$ such that
	\[ 2\frac{M}{k} + \frac{M}{l} + 2 \left( (s+r-2) \frac{M}{k} + (t-1)\frac{M}{l}\right) \equiv 0 \mod 2M, \]
	where we used the assumption that $M/l$ is even. Moreover, since the equation is modulo $2M$, it is clear that we can take $s, r, t$ such that $s, r \in \E{k}$ and $t \in \E{l}$. Dividing the left-hand side by $M$, we thus find that $\frac{2}{k} + \frac{1}{l} + \frac{2(s + r - 2)}{k} + \frac{2(t - 1)}{l} = \frac{2s + 2r - 2}{k} + \frac{2t - 1}{l}$ is an even integer. From this it follows that
	\[ R_{ks}(-1) \cdot R_{kr}(-1) \cdot R_{lt}(-1) = e^{i\left( \frac{\pi}{k} + s\frac{2\pi}{k} \right)} e^{i\left( \frac{\pi}{k} + r\frac{2\pi}{k} \right)} e^{i\left( \frac{\pi}{l} + t\frac{2\pi}{l} \right)} = e^{i \pi \left(\frac{2s+2r-2}{k} + \frac{2t-1}{l} \right)} = 1.\]
\end{proof}

\begin{lemma}
	\label{lem:prodRootsEqual1DiffSign}
	Take any positive integers $k, l > 0$ and write $M = \lcm(k,l)$. Then the following are true:
	\begin{enumerate}[label = (\roman*)]
		\item \label{item:prodRootsEq1Case3} If $M/k$ is even, then then there exist integers $s \in \E{k}$ and $t \in \E{l}$ such that
		\[ R_{ks}(-1) \cdot R_{lt}(1) = 1.\]
		\item \label{item:prodRootsEq1Case4} If $M/k$ is odd, then then there exist integers $s, r \in \E{k}$ and $t \in \E{l}$ such that
		\[ R_{ks}(-1) \cdot R_{kr}(-1) \cdot R_{lt}(1) = 1.\]
	\end{enumerate}
\end{lemma}
\begin{proof}
	For \ref{item:prodRootsEq1Case3}, note that $M/k$ and $M/l$ are coprime and thus by the Theorem of Bachet-B\'ezout, there exist integers $s,r, t \in \Z$ such that
	\[ \frac{M}{k} + 2 \left( (s-1) \frac{M}{k} + (t-1) \frac{M}{l}\right) \equiv 0 \mod 2M, \]
	where we used the assumption that $M/k$ is even. Moreover, since the equation is modulo $2M$, it is clear that we can take $s, t$ such that $s \in \E{k}$ and $t \in \E{l}$. Dividing the left-hand side by $M$, we thus find that $\frac{1}{k} + \frac{2(s-1)}{k} + \frac{2(t-1)}{l} = \frac{2s-1}{k} + \frac{2t-2}{l}$ is an even integer. From this it follows that
	\[ R_{ks}(-1) \cdot R_{lt}(1) = e^{i\left( \frac{\pi}{k} + (s-1)\frac{2\pi}{k} \right)} e^{i (t-1)\frac{2\pi}{l}} = e^{i\pi\left( \frac{2s-1}{k} + \frac{2t-2}{l} \right)} = 1.\]
	For \ref{item:prodRootsEq1Case4}, note that $M/k$ and $M/l$ are coprime and thus by the Theorem of Bachet-B\'ezout, there exist integers $s,r,t \in \Z$ such that
	\[ 2\frac{M}{k} + 2 \left( (s+r-2) \frac{M}{k} + (t-1) \frac{M}{l}\right) \equiv 0 \mod 2M. \]
	Moreover, since the equation is modulo $2M$, it is clear that we can take $s, r, t$ such that $s, r \in \E{k}$ and $t \in \E{l}$. Dividing the left-hand side by $M$, we thus find that $\frac{2}{k} + \frac{2(s + r - 2)}{k} + \frac{2(t-1)}{l} = \frac{2s - 2r - 2}{k} + \frac{2t - 2}{l}$ is an even integer. From this it follows that
	\[ R_{ks}(-1) \cdot R_{kr}(-1) \cdot R_{lt}(1) = e^{i\left( \frac{\pi}{k} + (s-1)\frac{2\pi}{k} \right)} e^{i\left( \frac{\pi}{k} + (r-1)\frac{2\pi}{k} \right)} e^{i (t-1)\frac{2\pi}{l}} = e^{i\pi\left( \frac{2s - 2r - 2}{k} + \frac{2t - 2}{l} \right)} = 1.\]
\end{proof}

In combination with the two preceding lemma's, the next lemma gives us an answer to a more general problem: given integers $k, l, n, m > 0$ and complex numbers $\alpha_1, \ldots, \alpha_n, \beta_1, \ldots, \beta_m \in \C^\ast$ with $\displaystyle \prod_{i = 1}^n \alpha_i = \pm 1$ and $\displaystyle \prod_{i = 1}^m \beta_i = \pm 1$, how can we obtain 1 as a product of $k$-roots of the numbers $\alpha_i$ and $l$-roots of the numbers $\beta_i$, with as little factors as possible, but with at least one $k$-root of $\alpha_i$ for any $i \in \E{n}$ and at least one $l$-root of $\beta_i$ for any $i \in \E{m}$.

\begin{lemma}
	\label{lem:rootsOfProdIsProdOfRoots}
	Let $n > 0$, $k > 0$ be positive integers and $\alpha_1,\ldots, \alpha_n \in \C \setminus \{0\}$. For any integer $s \in \E{k}$, there exist integers $s_1, \ldots, s_n \in \E{k}$ such that
	\[ R_{ks}\left( \prod_{i = 1}^n \alpha_i \right) = \prod_{i = 1}^n R_{ks_i}(\alpha_i). \]
\end{lemma}
\begin{proof}
	Note that all elements in the set
	\[ S = \left\{ \prod_{i = 1}^n R_{k s_i}(\alpha_i) \:\middle| \: s_1, \ldots, s_n \in \E{k}  \right\} \]
	are indeed $k$-roots of $\prod_{i = 1}^n \alpha_i$ since
	\[ \left( \prod_{i = 1}^n R_{k s_i}(\alpha_i) \right)^k = \prod_{i = 1}^n R_{k s_i}(\alpha_i)^k = \prod_{i = 1}^n \alpha_i. \]
	By only varying $s_1$ in $\E{k}$, it is easily verified that $S$ must count at least $k$ elements. We can thus conclude that $S$ is exactly equal to the set of $k$-roots of $\prod_{i = 1}^n \alpha_i$. From this observation, the lemma follows immediately. 
\end{proof}

We are now ready to prove the upper bound on the $R_\infty$-nilpotency index of $A(\Gamma)$ of which we recall the definition:
\begin{equation*}
	\Xi(\Gamma) = \min \left\{ c(\lambda, \mu) \: \middle| \: \lambda, \mu \in \Lambda, \, \{\lambda, \mu \} \in \overline{E} \right\}
\end{equation*}
where for any $\lambda, \mu \in \Lambda$:
\begin{equation*}
	c(\lambda, \mu) = \begin{cases}
		\max\{ 2|\lambda| + |\mu|, |\lambda| + 2|\mu| \} \quad & \text{if } \lambda \neq \mu\\
		2|\lambda| & \text{if } \lambda = \mu.
	\end{cases}
\end{equation*}

\begin{theorem}
	Let $\Gamma = (V, E)$ be a non-empty graph. The group $A(\Gamma, \Xi(\Gamma))$ has the $R_\infty$-property.
\end{theorem}

\begin{proof}
	Let $\overline{\Gamma} = (\Lambda, \overline{E}, \Phi)$ be the quotient graph of $\Gamma$. Take $\lambda, \mu \in \Lambda$ arbitrarily such that $\{\lambda, \mu\} \in \overline{E}$. It suffices to prove that $A(\Gamma, c)$ has the $R_\infty$-property with $c = c(\lambda, \mu)$ (since we can then take the minimum of $c(\lambda, \mu)$ for all $\{\lambda, \mu\} \in \overline{E}$ in the end).

	Take an arbitrary automorphism $\varphi \in \Aut(A(\Gamma, c))$. By Lemma \ref{lem:reidemeisterNumberFinGenNilpGroup} it satisfies to prove that the induced automorphism $\overline{\varphi} \in \Aut_g(L^\C (\Gamma, c))$ has an eigenvalue 1. Using Lemma \ref{lem:ReducedFormOfInducedAutos}, it is sufficient to prove that for any automorphism $\psi \in \Aut(\overline{\Gamma})$ with disjoint cycle decomposition $\psi = \sigma_1 \circ \ldots \circ \sigma_d$ and any tuples $\alpha_i = (\alpha_{i1}, \ldots, \alpha_{i n_i}) \in \C^{n_i}$ where $n_i$ is the size of the coherent components in the cycle $\sigma_i$ and $\alpha_{i1} \cdot \ldots \cdot \alpha_{in_i} = \pm 1$, it holds that the automorphism $\pi_c^{-1}(T(\sigma_1, \alpha_1) \cdot \ldots \cdot T(\sigma_d, \alpha_d))$ has an eigenvalue 1.
	
	Note that every coherent component lies in a unique cycle from the decomposition of $\psi$. Let $p,q \in \E{d}$ be the indices such that $c_p$ and $c_q$ denote these cycles containing $\lambda$ and $\mu$, respectively. For sake of notation, we will write $\sigma := \sigma_p$, $\tau := \sigma_q$, $n = |\lambda|$, $m = |\mu|$, $\alpha = (\alpha_1, \ldots, \alpha_n) := \alpha_p$ and $\beta = (\beta_1, \ldots, \beta_m) := \alpha_q$ and let $k$, $l$ denote the orders of the cycles $\sigma$ and $\tau$, respectively. We divide the proof into several cases:
	\begin{itemize}
		\item \underline{$\sigma = \tau$}
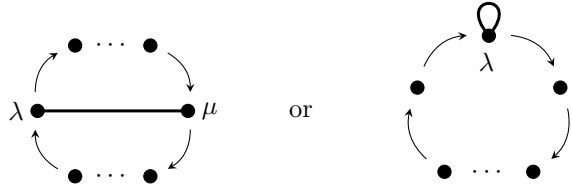
		\begin{figure}[H]
			\centering
			\picOne
		\end{figure}

		\noindent Note that in this case also $n = m$, $\alpha = \beta$ and $k = l$. 
		\begin{itemize}
			\item First consider the case where $n = m = 1$. Then we must have that $\lambda \neq \mu$, $c(\lambda, \mu) = 3$ and $T(\sigma, \alpha) = T(\sigma, (\pm 1))$. If $\alpha = 1$, then there exists an eigenvector with eigenvalue 1 in $L^\C_1(\Gamma, c)$ by the same argument as in \ref{item:transFreeCase1} from section \ref{sec:transposition-freeGraphs}. If, on the other hand $\alpha = -1$, then there exists an eigenvector with eigenvalue 1 in $L^\C_2(\Gamma, c)$ by the same argument as in \ref{item:transFreeCase2} from section \ref{sec:transposition-freeGraphs}.

			\item Next, consider the case where $n = m > 1$. Using that $\prod_{i = 1}^n \alpha_i = \pm 1$, we can combine Lemma \ref{lem:prodRootsEqual1SameSign} \ref{item:prodRootsEq1Case0} \ref{item:prodRootsEq1Case1} (with $k = l$) and \ref{lem:rootsOfProdIsProdOfRoots} to get the existence of indices $s_1, \ldots, s_n, t_1, \ldots, t_n \in \E{k}$ such that 
			\[ \left( \prod_{i = 1}^n R_{k s_i}(\alpha_i) \right) \left( \prod_{i = 1}^n R_{k t_i}(\alpha_i) \right) = 1.\]
			Next, consider the set $A := (\lambda \setminus \{v_{\lambda 1}\}) \cup \{ v_{\mu 1} \}$ for which it is clear that $|A| \geq 2$. From the fact that $\{\lambda, \mu\} \in \overline{E}$ we have that $\{v, w\} \in E$ for any $v \in \lambda, w \in \mu$, hence $A$ is connected. Define
			\[W = \{ w(\sigma, \alpha)_{i s_i}, \, w(\sigma, \alpha)_{i t_i} \mid 1 \leq i \leq n \}\]
			where we remind the reader of the definition of $w(\sigma, \alpha)_{ij} \in \Sp_\C(V)$ in \eqref{eq:eigenvectOfT}. Define a map
			\[\kappa: W \to A: w(\sigma, \alpha)_{ij} \mapsto \begin{cases}
				v_{\mu i} &\text{ if } $i = 1$\\
				v_{\lambda i} &\text{ else.} 
			\end{cases}\]
			It is clear that $A$, $W$ and $\kappa$ satisfy the assumptions of Lemma \ref{lem:existenceNonZeroBWWithGivenWeightGeneralization}. Define the weight
			\[ e: W \to \N: w(\sigma, \alpha)_{ij} \mapsto \begin{cases}
				1 &\text{ if } s_i \neq t_i\\
				2 &\text{ else,}
			\end{cases}.\]
			for which it holds that
			\[\sum_{w \in W} e(w) = 2n \leq c(\lambda, \mu) = c.\]
			By consequence we can apply Lemma \ref{lem:existenceNonZeroBWWithGivenWeightGeneralization} to $A$, $W$ and $\kappa$ for the weight $e$ and get a bracket word $b \in \BW(W)$ of weight $e$ such that $\phi_W^c(b) \in L^\C (\Gamma, c)$ is non-zero. Since each element of $W$ is an eigenvector of the automorphism $\pi_c^{-1}(T(\sigma_1, \alpha_1) \cdot \ldots \cdot T(\sigma_d, \alpha_d))$, it follows that $\phi_W^c(b)$ is still an eigenvector. By construction, it is clear that the eigenvalue of $\phi_W^c(b)$ is equal to $1$.
		\end{itemize}
		
		\item \underline{$\sigma \neq \tau$.}
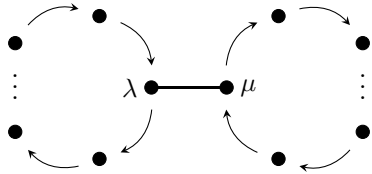
		\begin{figure}[H]
			\centering
			\picTwo
		\end{figure}
		\noindent Note that since $\{\lambda, \mu\} \in \overline{E}$, the set $A := \lambda \cup \mu$ is connected and that since $\lambda \neq \mu$ we have $|A| \geq 2$. By combining Lemmas \ref{lem:prodRootsEqual1SameSign}, \ref{lem:prodRootsEqual1DiffSign} and \ref{lem:rootsOfProdIsProdOfRoots}, we know one of following three cases is true, depending on the integers $k$ and $l$ and the signs of $\prod_{i = 1}^n \alpha_i$ and $\prod_{i = 1}^m \beta_i$. This dependence is summarized in Table \ref{tab:proofUpperBoundCases} below.
		\begin{enumerate}[label = P\arabic*.]
			\item \label{item:proofUpperBoundCase1} There exist $s_1, \ldots, s_n \in \E{k}, t_1, \ldots, t_m \in \E{l}$ such that
			\[ \left( \prod_{i = 1}^n R_{k s_i}(\alpha_i) \right) \left( \prod_{i = 1}^m R_{l t_i}(\beta_i) \right) = 1. \]
			In this case, define $W := \{ w(\sigma, \alpha)_{i s_i} \mid 1 \leq i \leq n \} \cup \left\{ w(\tau, \beta)_{i t_i} \mid 1 \leq i \leq m \right\}$ and $\kappa: W \to A: w(\sigma, \alpha)_{ij} \mapsto v_{\lambda i}, \, w(\tau, \beta)_{ij} \mapsto v_{\mu i}$. It is clear that $A$, $W$ and $\kappa$ satisfy the assumptions of Lemma \ref{lem:existenceNonZeroBWWithGivenWeightGeneralization}. Define the constant weight $e:W \to \N: w \mapsto 1$ for which we have \[\sum_{w \in W} e(w) = n + m \leq c(\lambda, \mu) = c.\]
			By consequence we can apply Lemma \ref{lem:existenceNonZeroBWWithGivenWeightGeneralization} to $A$, $W$ and $\kappa$ for the weight $e$ and get a bracket word $b \in \BW(W)$ of weight $e$ such that $\phi_W^c(b) \in L^\C (\Gamma, c)$ is non-zero. Since each element of $W$ is an eigenvector of the automorphism $\pi_c^{-1}(T(\sigma_1, \alpha_1) \cdot \ldots \cdot T(\sigma_d, \alpha_d))$, it follows that $\phi_W^c(b)$ is still an eigenvector. By construction, it is clear that the eigenvalue of $\phi_W^c(b)$ is equal to $1$.
			\item \label{item:proofUpperBoundCase2} There exist $s_1, \ldots, s_n, r_1, \ldots, r_n \in \E{k}, t_1, \ldots, t_m \in \E{l}$ such that
			\[ \left( \prod_{i = 1}^n R_{k s_i}(\alpha_i) \right) \left( \prod_{i = 1}^n R_{k r_i}(\alpha_i) \right) \left( \prod_{i = 1}^m R_{l t_i}(\beta_i) \right) = 1. \]
			In this case, define $W := \{ w(\sigma, \alpha)_{i s_i}, \, w(\sigma, \alpha)_{i r_i} \mid 1 \leq i \leq n \} \cup \left\{ w(\tau, \beta)_{i t_i} \mid 1 \leq i \leq m \right\}$ and $\kappa: W \to A: w(\sigma, \alpha)_{ij} \mapsto v_{\lambda i}, \, w(\tau, \beta)_{ij} \mapsto v_{\mu i}$. It is clear that $A$, $W$ and $\kappa$ satisfy the assumptions of Lemma \ref{lem:existenceNonZeroBWWithGivenWeightGeneralization}. Define the weight
			\[ e:W \to \N: w(\tau, \beta)_{ij} \mapsto 1, \, w(\sigma, \alpha)_{ij} \mapsto \begin{cases}
				1 &\text{ if } s_i \neq r_i\\
				2 &\text{ else,}
			\end{cases}  \]
			for which it holds that
			\[ \sum_{w \in W} e(w) = 2n + m \leq c(\lambda, \mu) = c. \]
			By consequence we can apply Lemma \ref{lem:existenceNonZeroBWWithGivenWeightGeneralization} to $A$, $W$ and $\kappa$ for the weight $e$ and get a bracket word $b \in \BW(W)$ of weight $e$ such that $\phi_W^c(b) \in L^\C (\Gamma, c)$ is non-zero. Since each element of $W$ is an eigenvector of the automorphism $\pi_c^{-1}(T(\sigma_1, \alpha_1) \cdot \ldots \cdot T(\sigma_d, \alpha_d))$, it follows that $\phi_W^c(b)$ is still an eigenvector. By construction, it is clear that the eigenvalue of $\phi_W^c(b)$ is equal to $1$.
			
			\item \label{item:proofUpperBoundCase3} There exist $s_1, \ldots, s_n \in \E{k}, t_1, \ldots, t_m, r_1, \ldots, r_n \in \E{l}$ such that
			\[ \left( \prod_{i = 1}^n R_{k s_i}(\alpha_i) \right) \left( \prod_{i = 1}^m R_{l t_i}(\beta_i) \right) \left( \prod_{i = 1}^m R_{l r_i}(\beta_i) \right) = 1. \]
			In this case, define $W := \{ w(\sigma, \alpha)_{i s_i} \mid 1 \leq i \leq n \} \cup \left\{ w(\tau, \beta)_{i t_i}, \, w(\tau, \beta)_{i r_i} \mid 1 \leq i \leq m \right\}$ and $\kappa: W \to A: w(\sigma, \alpha)_{ij} \mapsto v_{\lambda i}, \, w(\tau, \beta)_{ij} \mapsto v_{\mu i}$. It is clear that $A$, $W$ and $\kappa$ satisfy the assumptions of Lemma \ref{lem:existenceNonZeroBWWithGivenWeightGeneralization}. Define the weight
			\[ e:W \to \N: w(\sigma, \alpha)_{ij} \mapsto 1, \, w(\tau, \beta)_{ij} \mapsto \begin{cases}
				1 &\text{ if } t_i \neq r_i\\
				2 &\text{ else}
			\end{cases} \]
			for which it holds that
			\[ \sum_{w \in W} e(w) = n + 2m \leq c(\lambda, \mu) = c. \]
			By consequence we can apply Lemma \ref{lem:existenceNonZeroBWWithGivenWeightGeneralization} to $A$, $W$ and $\kappa$ for the weight $e$ and get a bracket word $b \in \BW(W)$ of weight $e$ such that $\phi_W^c(b) \in L^\C (\Gamma, c)$ is non-zero. Since each element of $W$ is an eigenvector of the automorphism $\pi_c^{-1}(T(\sigma_1, \alpha_1) \cdot \ldots \cdot T(\sigma_d, \alpha_d))$, it follows that $\phi_W^c(b)$ is still an eigenvector. By construction, it is clear that the eigenvalue of $\phi_W^c(b)$ is equal to $1$.
		\end{enumerate}
	\end{itemize}
	This concludes the proof.
\end{proof}

\begin{table}[H]
	\centering
	\begin{tabular}{l|l|l|l|l}
		& $\begin{array}{l}
			M/k \text{ even}\\
			M/l \text{ even}
		\end{array}$ & $\begin{array}{l}
		M/k \text{ odd}\\
		M/l \text{ odd}
	\end{array}$ & $\begin{array}{l}
	M/k \text{ even}\\
	M/l \text{ odd}
\end{array}$ & $\begin{array}{l}
M/k \text{ odd}\\
M/l \text{ even}
\end{array}$ \\ \hline
		$\arraycolsep=1.4pt\def\arraystretch{1.5}
		\begin{array}{l}
			\prod_i^n \alpha_i = 1\\
			\prod_i^n \beta_i = 1
		\end{array}$ & $\arraycolsep=1.4pt\def\arraystretch{1.5}\begin{array}{l}
		\text{Lemma } \ref{lem:prodRootsEqual1SameSign}\,\ref{item:prodRootsEq1Case0}\\
		\rightarrow \ref{item:proofUpperBoundCase1}
	\end{array}$                                                                                                             & $\arraycolsep=1.4pt\def\arraystretch{1.5}\begin{array}{l}
	\text{Lemma } \ref{lem:prodRootsEqual1SameSign}\,\ref{item:prodRootsEq1Case0}\\
	\rightarrow \ref{item:proofUpperBoundCase1}
\end{array}$ & $\arraycolsep=1.4pt\def\arraystretch{1.5}\begin{array}{l}
\text{Lemma } \ref{lem:prodRootsEqual1SameSign}\,\ref{item:prodRootsEq1Case0}\\
\rightarrow \ref{item:proofUpperBoundCase1}
\end{array}$ & $\arraycolsep=1.4pt\def\arraystretch{1.5}\begin{array}{l}
\text{Lemma } \ref{lem:prodRootsEqual1SameSign}\,\ref{item:prodRootsEq1Case0}\\
\rightarrow \ref{item:proofUpperBoundCase1}
\end{array}$ \\ \hline
		$\arraycolsep=1.4pt\def\arraystretch{1.5}\begin{array}{l}
			\prod_i^n \alpha_i = -1\\
			\prod_i^n \beta_i = -1
		\end{array}$ & $\arraycolsep=1.4pt\def\arraystretch{1.5}\begin{array}{l}
		\text{Lemma } \ref{lem:prodRootsEqual1SameSign}\,\ref{item:prodRootsEq1Case1}\\
		\rightarrow \ref{item:proofUpperBoundCase1}
	\end{array}$ & $\arraycolsep=1.4pt\def\arraystretch{1.5}\begin{array}{l}
	\text{Lemma } \ref{lem:prodRootsEqual1SameSign}\,\ref{item:prodRootsEq1Case1}\\
	\rightarrow \ref{item:proofUpperBoundCase1}
\end{array}$ & $\arraycolsep=1.4pt\def\arraystretch{1.5}\begin{array}{l}
\text{Lemma } \ref{lem:prodRootsEqual1SameSign}\,\ref{item:prodRootsEq1Case2}\\
\rightarrow \ref{item:proofUpperBoundCase3}
\end{array}$ & $\arraycolsep=1.4pt\def\arraystretch{1.5}\begin{array}{l}
\text{Lemma } \ref{lem:prodRootsEqual1SameSign}\,\ref{item:prodRootsEq1Case2}\\
\rightarrow \ref{item:proofUpperBoundCase2}
\end{array}$ \\ \hline
		$\arraycolsep=1.4pt\def\arraystretch{1.5}\begin{array}{l}
			\prod_i^n \alpha_i = -1\\
			\prod_i^n \beta_i = 1
		\end{array}$ & $\arraycolsep=1.4pt\def\arraystretch{1.5}\begin{array}{l}
		\text{Lemma } \ref{lem:prodRootsEqual1DiffSign}\,\ref{item:prodRootsEq1Case3}\\
		\rightarrow \ref{item:proofUpperBoundCase1}
	\end{array}$ & $\arraycolsep=1.4pt\def\arraystretch{1.5}\begin{array}{l}
	\text{Lemma } \ref{lem:prodRootsEqual1DiffSign}\,\ref{item:prodRootsEq1Case4}\\
	\rightarrow \ref{item:proofUpperBoundCase2}
\end{array}$ & $\arraycolsep=1.4pt\def\arraystretch{1.5}\begin{array}{l}
\text{Lemma } \ref{lem:prodRootsEqual1DiffSign}\,\ref{item:prodRootsEq1Case3}\\
\rightarrow \ref{item:proofUpperBoundCase1}
\end{array}$ & $\arraycolsep=1.4pt\def\arraystretch{1.5}\begin{array}{l}
\text{Lemma } \ref{lem:prodRootsEqual1DiffSign}\,\ref{item:prodRootsEq1Case4}\\
\rightarrow \ref{item:proofUpperBoundCase2}
\end{array}$ \\ \hline
		$\arraycolsep=1.4pt\def\arraystretch{1.5}\begin{array}{l}
			\prod_i^n \alpha_i = 1\\
			\prod_i^n \beta_i = -1
		\end{array}$ & $\arraycolsep=1.4pt\def\arraystretch{1.5}\begin{array}{l}
		\text{Lemma } \ref{lem:prodRootsEqual1DiffSign}\,\ref{item:prodRootsEq1Case3}\\
		\rightarrow \ref{item:proofUpperBoundCase1}
	\end{array}$ & $\arraycolsep=1.4pt\def\arraystretch{1.5}\begin{array}{l}
	\text{Lemma } \ref{lem:prodRootsEqual1DiffSign}\,\ref{item:prodRootsEq1Case4}\\
	\rightarrow \ref{item:proofUpperBoundCase3}
\end{array}$ & $\arraycolsep=1.4pt\def\arraystretch{1.5}\begin{array}{l}
\text{Lemma } \ref{lem:prodRootsEqual1DiffSign}\,\ref{item:prodRootsEq1Case4}\\
\rightarrow \ref{item:proofUpperBoundCase3}
\end{array}$ & $\arraycolsep=1.4pt\def\arraystretch{1.5}\begin{array}{l}
\text{Lemma } \ref{lem:prodRootsEqual1DiffSign}\,\ref{item:prodRootsEq1Case3}\\
\rightarrow \ref{item:proofUpperBoundCase1}
\end{array}$
	\end{tabular}
	\caption{Summary of which lemma is used in combination with Lemma \ref{lem:rootsOfProdIsProdOfRoots} to obtain one of the cases \ref{item:proofUpperBoundCase1}, \ref{item:proofUpperBoundCase2} or \ref{item:proofUpperBoundCase3}}
	\label{tab:proofUpperBoundCases}
\end{table}

\bibliography{ref}
\bibliographystyle{plain}

\end{document}